\documentclass[12pt,reqno]{amsart}
\usepackage{graphicx}
\usepackage{amssymb}
\usepackage{amsfonts}
\usepackage{amsmath,latexsym,amssymb,amsfonts,amsbsy, amsthm}
\usepackage[usenames]{color}
\usepackage{xspace,colortbl}
\usepackage{mathrsfs}
\usepackage[colorlinks,linkcolor=blue,citecolor=blue]{hyperref}
\usepackage[titletoc]{appendix}
\usepackage{pdfsync}
\usepackage{esint}
\allowdisplaybreaks

\newcommand{\beq}{\begin{equation}}
	\newcommand{\eeq}{\end{equation}}
\newcommand{\ben}{\begin{eqnarray}}
	\newcommand{\een}{\end{eqnarray}}
\newcommand{\beno}{\begin{eqnarray*}}
	\newcommand{\eeno}{\end{eqnarray*}}
\newcommand{\R}{\mathbb{R}}
\newtheorem{thm}{Theorem}[section]
\newtheorem{defi}[thm]{Definition}
\newtheorem{lem}[thm]{Lemma}
\newtheorem{prop}[thm]{Proposition}
\newtheorem{coro}[thm]{Corollary}
\newtheorem{rmk}[thm]{Remark}

\title[Blow up analysis for MEMS, I]{Blow up analysis for a parabolic MEMS problem, I: H\"{o}lder estimate}

\author[K. Wang]{Kelei Wang$^\dag$}
\address{$^\dag$School of Mathematics and Statistics \\ Wuhan University\\
	Wuhan 430072, China}
\email{wangkelei@whu.edu.cn}

\author[G. Yi]{Guangzeng Yi$^\ast$}
\address{$\ast$Department of Mathematics and Statistics\\ University of Jyv\"askyl\"a,
P.O. Box 35 (MaD)\\
FI-40014 University of Jyv\"askyl\"a\\
Finland}
\email{guangzeng.m.yi@jyu.fi}

\date{\today}

\thanks{K. Wang was supported by  National Key R\&D Program of China (No. 2022YFA1005602) and the National Natural Science Foundation of China (No. 12131017 and No. 12221001). }



\begin{document}

\begin{abstract}
This is the first in a series of papers devoted to the blow up analysis for the quenching phenomena in a parabolic MEMS equation. In this paper, we first give an optimal H\"{o}lder estimate for solutions to this equation by using the blow up method and some Liouville theorems on stationary two-valued caloric functions, and then establish a convergence theory for sequences of uniformly H\"{o}lder continuous solutions. These results are also used to prove a stratification theorem on the rupture set $\{u=0\}$.
\end{abstract}

\maketitle

\tableofcontents

\section{Introduction}\label{sec introduction}
\setcounter{equation}{0}
	
	Micro-electro-mechanical systems (MEMS)  are widely used devices in engineering and technology. The modelling of MEMS poses several interesting partial differential equations. Rigorous mathematical analysis of these PDEs has attracted a lot of attention in the past two decades.  For more backgrounds, we refer the reader to the monograph of Esposito-Ghoussoub-Guo \cite{Esposito-Ghoussoub-Guo} and the survey paper of  Lauren\c{c}ot-Walker \cite{Laurencot-Walker2017},  where an extensive list of results and references about this problem can also be found. Some more recent studies on these problems can be found in \cite{Esteve-Souplet2018, Esteve-Souplet2019, Laurencot-Walker2018}.
	
	An important feature of these PDEs is that their solutions could exhibit \emph{the quenching phenomena}. Consider for example the second order parabolic MEMS equation
	\begin{equation}\label{eqn}
		\partial_tu-\Delta u=-u^{-p}
	\end{equation}
	in $\Omega_T=\Omega\times (0,T)$, where $p>1$, $\Omega$ is a smooth domain in $\R^n$ (which could be unbounded). A Cauchy-Dirichlet boundary condition $u\lfloor_{\partial^p \Omega_T}=\varphi$ is imposed on the parabolic boundary $\partial^p\Omega_T$,  where $\varphi>0$ belongs to $C^{\alpha,\alpha/2}(\partial^p \Omega_T)$, with
	\[\alpha:=2/(p+1).\]
	This Cauchy-Dirichlet problem has a unique, positive solution, which however could exist only locally in time. 
	This is because, due to the  effect of  the nonlinearity $-u^{-p}$, the solution could touch down to $0$ at some time, where $u$ becomes singular. This is \emph{the quenching phenomena} in MEMS problem.
	
	In some special cases, e.g. when the spatial dimension is one or when the initial-boundary value $\varphi$ takes some special forms (e.g. being a constant),  there are several results describing the nature of this quenching phenomena, see  \cite[Chapter 10]{Esposito-Ghoussoub-Guo}. However, to the best of our knowledge, the analysis of this phenomena in  general  is not done yet. In this series of papers, we will develop a blow up analysis approach to this problem, which  allows us to give a  precise description of how solutions to \eqref{eqn} approach $0$ in  finite time, and the structure of the rupture set $\{u=0\}$.
	
	In this first paper, we will establish an optimal H\"{o}lder estimate for solutions to \eqref{eqn},  and use this estimate to give a convergence theory for sequences of solutions to \eqref{eqn}. These two results are used in this paper to estimate the Hausdorff dimension of the rupture set $\{u=0\}$. These results will also be used in sequel papers for further analysis of the quenching phenomena in \eqref{eqn}.

	\section{Main results}\label{section main result}
	\setcounter{equation}{0}

	We  assume $u$ is a local solution to the Cauchy-Dirichlet problem of \eqref{eqn} in $\Omega_T$, where $T>0$ is the first quenching moment, i.e., $u(x,t)>0$ for any $(x,t)\in\Omega\times (0,T)$, while
	\[\lim_{t\to T^-}\inf_{x\in\Omega}u(x,t)=0.\]
	The first main result of this paper is a global H\"{o}lder estimate on $u$. In the following $C^{\alpha,\alpha/2}(\overline{\Omega_T})$ denotes the space of $\alpha$-H\"{o}lder continuous functions with respect to the parabolic distance.
	\begin{thm}\label{thm 1.1}
		If $u$ is the smooth solution of the Cauchy-Dirichlet problem for \eqref{eqn} with positive initial boundary value $\varphi\in C^{\alpha,\alpha/2}(\partial^p\Omega_T)$, then $u\in C^{\alpha,\alpha/2}(\overline{\Omega_T})$.
	\end{thm}
	Notice that  \eqref{eqn} is scaling invariant in the sense that if $u$ is a solution, then for any $\lambda>0$,
	\[u_\lambda(x,t):=\lambda^{-\alpha}u(\lambda x, \lambda^2t)\]
	is also a solution.
	The exponent $\alpha=2/(p+1)$ is compatible with this scaling invariance. This H\"{o}lder estimate is optimal in view of some explicit solutions of \eqref{eqn} such as
	\[u_0(t):= (p+1)^{\frac{1}{p+1}} (-t)^{\frac{1}{p+1}}, \quad \text{for}~ t<0,\]
	and
	\[u(x)=\left[\frac{2}{p+1}\left(n-2+\frac{2}{p+1}\right)\right]^{-\frac{1}{p+1}}|x|^{\frac{2}{p+1}}, \quad \text{for} ~ x\in\R^n, ~ n\geq 2.\]

	By this H\"{o}lder estimate, we can define the trace of $u$ at time $T$, $u(x,T)$, by continuation. In particular, the rupture set
	\[\Sigma_T:=\{x:u(x,T)=0\}\]
	is a closed subset of $\overline{\Omega}$. A direct consequence of this theorem is, if the boundary value $\varphi$ is strictly positive, then there is no rupture point on the boundary. This was previously  known only in the case when $\Omega$ is convex, see \cite[Theorem 8.3.1]{Esposito-Ghoussoub-Guo}.
	\begin{coro}\label{coro 1.2}
		If $\varphi(x,T)>0$ on $\partial\Omega$, then $\Sigma_T$ is compactly supported in $\Omega$.
	\end{coro}
	
	Theorem \ref{thm 1.1}  corresponds to the H\"{o}lder estimate for the elliptic version of \eqref{eqn},
 \[\Delta u=u^{-p},\]
which was established in Davila-Wang-Wei \cite{Davila-Wang-Wei}. The main idea is to use the blow up method, and then reduce the estimate to some Liouville properties for entire harmonic functions or \emph{stationary two-valued harmonic functions} in the elliptic case,  and  for ancient solutions to the heat equation or \emph{stationary two-valued caloric functions} in the parabolic case.  
\begin{defi}\label{stationary two-valued caloric fct}[Stationary two-valued  caloric functions]
		Given a domain $\Omega\subset\R^{n+1}$,
		a nonnegative function $u\in L^1_{loc}(\Omega)$ is called a stationary two-valued caloric function if
		\begin{description}
			\item [(i)] $\nabla u, \partial_tu\in L^2_{loc}(\Omega)$;
			\item [(ii)] $\Delta u-\partial_tu\geq0$ in the distributional sense;
			\item [(iii)] $u(\partial_tu-\Delta u)=0$ in the sense that, for any $\eta\in C_0^\infty(\Omega)$,
			\[\int_{\Omega} \partial_tu u\eta^2+|\nabla u|^2\eta^2+2\eta u \nabla u\cdot \nabla\eta=0;\]

			\item [(iv)] for any vector field $Y \in C_0^{\infty}(\Omega,\R^{n})$,
			\begin{equation}\label{stationary condition 2}
				\int_{\Omega}|\nabla u|^{2}divY - 2DY(\nabla u, \nabla u) - 2 \partial_tu \nabla u \cdot Y = 0;
			\end{equation}
			
			\item [(v)] $u$ satisfies the localized energy inequality, that is, for any $\eta\in C_0^{\infty}(\Omega)$,
			\begin{equation}\label{localized energy inequality 2}
            \begin{split}
                \frac{d}{dt}\int_{\mathbb{R}^{n}} |\nabla u|^2 \eta^2 \leq&-2\int_{\R^{n}}|\partial_tu|^2 \eta^2+2\int_{\R^{n}} |\nabla u|^2 \eta \partial_t\eta\\&- 4\int_{\R^{n}}\partial_tu\eta \nabla u \cdot \nabla\eta.
            \end{split}
			\end{equation}
		\end{description}
  \begin{rmk}
     \begin{itemize}
         \item Although we only assume $u$ to be $L^1_{loc}$, by the subcaloric property (ii) and because $u$ is nonnegative, $u\in L^\infty_{loc}$. 
         \item The elliptic version of the above notion is stationary two-valued  harmonic function. This is a special case of Almgren's $Q$-valued functions, see \cite{Almgren-big} and \cite{DeLellis-Spadaro}. Although in these works only Dirichlet energy minimizers are considered, recently it is found that stationary ones arise in various settings, see \cite{Bouafia,Bouafia-DePauw-Wang,hirsch2022interior,krummel2013fine,lee2023lebesgue,Lin2014}.
     \end{itemize} 
  \end{rmk}
	\end{defi}

	The method in \cite{Davila-Wang-Wei} is  inspired by
	the blow up method in  Noris-Tavares-Terracini-Verzini \cite{Terracini2010}, where  they give a uniform H\"{o}lder estimate for sequences of solutions to some singularly perturbed elliptic systems such as 
 \[\Delta u_{i,\beta}=\beta u_{i,\beta}\sum_{j\neq i} u_{j,\beta}^2, \quad i=1,\dots, N, \quad \beta\to+\infty.\]
See also Dancer-Wang-Zhang \cite{Dancer-Wang-Zhang} for the corresponding parabolic case of this singular perturbation problem.

The reason that  stationary two-valued caloric functions arise in this blow up argument is due to  the following observation. If the H\"{o}lder semi-norm of a solution $u$  is too large, then we  define a rescaling
\[\widetilde{u}(x,t):=L^{-1}\lambda^{-\frac{2}{p+1}} u(x_0+\lambda x, t_0+\lambda^2 t),\]
where $(x_0,t_0)$ is a well-chosen base point, $\lambda$ is a scaling parameter and $L$ is a large normalization constant which is chosen to make sure that the H\"{o}lder semi-norm of $\widetilde{u}$ is equal to $1$. (This normalization condition will lead to a contradiction with some Liouville properties.) The introduction of $L$  here of course destroys the scaling invariance of the original equation. In fact now the equation is changed to
\begin{equation}\label{1}
\partial_t\widetilde{u}-\Delta\widetilde{u}=-L^{-p-1}\widetilde{u}^{-p}.
\end{equation}
Because $L$ is large, we see that $\widetilde{u}$ is almost caloric in its positivity set $\{\widetilde{u}>0\}$, which will lead to the condition (ii) in Definition \ref{stationary two-valued caloric fct} after passing to the limit.  However, in order to establish Liouville theorems, we also need  \emph{the stationary condition} and  \emph{the localized energy identity}. This follows by taking limit in  a suitable \emph{stationary condition} and  \emph{localized energy identity} associated to \eqref{1}.

	This blow up method can also be used to establish an interior uniform H\"{o}lder estimate.
	\begin{thm}\label{thm 1.3}
		Suppose $(u_i)$ is a sequence of positive solutions of \eqref{eqn} in $Q_2^-:=B_2\times(-4,0]$, satisfying
		\begin{equation}\label{14}
			\sup\limits_{i}\int_{Q_2^-}u_i(x,t)<+\infty.
		\end{equation}
		Then there exists a constant $C$ independent of $i$ such that
		\[\|u_i\|_{C^{\alpha,\alpha/2}(\overline{Q_1^-})} \leq C.\]
	\end{thm}
	
	Next we discuss how a sequence of solutions to \eqref{eqn}	converge and what is their singular limit. We need some notions on weak solutions of \eqref{eqn}.
	\begin{defi}\label{def 1.4}
		A function $u$ is called an $L^1$ weak solution of \eqref{eqn} if $u\geq0$, $u\in L^1_{loc}$, $u^{-p}\in L^1_{loc}$, and it satisfies \eqref{eqn} in the distributional sense.
	\end{defi}
	
	\begin{defi}\label{def 1.5}
		A function $u$ (which, without loss of generality, is assumed to be defined in $Q_1:=B_1\times(-1,1)$) is called a suitable
		weak solution of \eqref{eqn} if
		\begin{description}
			\item [(i)]~$u$ is an $L^1$ weak solution;
			\item [(ii)]~$u$ is stationary in the sense that, for any vector field $Y\in C_0^{\infty}(Q_1,\mathbb R^n)$,
			\begin{equation}\label{stationary condition}
				\int \left(\frac{1}{2}|\nabla u|^2-\frac{u^{1-p}}{p-1}\right)\text{div} Y-DY(\nabla u,\nabla u)-\partial_tu(\nabla u\cdot Y)=0,
			\end{equation}
			where $\text{div}Y$ is the divergence of $Y$, and
			\[DY(\nabla u,\nabla u):=\sum_{i,j=1}^{n}\frac{\partial Y_i}{\partial x_j}\frac{\partial u}{\partial x_i}\frac{\partial u}{ \partial x_j};\]

			\item [(iii)]  $u$ satisfies the localized energy inequality, that is, for any $\eta\in C_0^{\infty}(\R^{n+1})$,
			\begin{equation}\label{localized energy inequality}
				\begin{split}
					\frac{\mathrm{d}}{\mathrm{d} t} \int_{\mathbb{R}^{n}} \left(\frac{|\nabla u|^2}{2}-\frac{u^{1-p}}{p-1}\right) \eta^2 \leq& -  \int_{\mathbb{R}^{n}}|\partial_tu|^2 \eta^2-2 \int_{\mathbb{R}^{n}} \partial_tu \left(\nabla u \cdot \nabla \eta\right)\eta\\
					&+2\int_{\mathbb{R}^{n}} \left(\frac{|\nabla u|^2}{2}-\frac{u^{1-p}}{p-1}\right)\eta \partial_t\eta.
				\end{split}
			\end{equation}
		\end{description}
	\end{defi}
	In the above, the energy inequality should be understood in the distributional sense. For positive smooth solutions, the stationary condition \eqref{stationary condition} can be obtained by multiplying (\ref{eqn}) by $\nabla u\cdot Y$ and integrating by parts, and the energy inequality (\ref{localized energy inequality}) becomes an identity, which can be obtained by testing (\ref{eqn}) with $\partial_tu\eta^2$ and integrating by parts.

	For any sequence of uniformly $\alpha$-H\"{o}lder continuous, suitable weak solutions, the following convergence result holds.
	\begin{thm}\label{thm 1.6}
		Let $(u_i)$ be a sequence of suitable weak solutions of \eqref{eqn} in $Q_2^-$ satisfying
		\begin{equation}\label{15}
			\sup_{i}\|u_i\|_{C^{\alpha,\alpha/2}(\overline {Q^-_{3/2}})}<+\infty.
		\end{equation}
		Then up to a subsequence,
		\begin{description}
			\item [(i)]  $u_i\to u_\infty$ uniformly in $Q_1^-$;
			\item [(ii)] $\nabla u_i\to\nabla u_{\infty}$ strongly in $L^2(Q_1^-)$, $\partial_tu_i\rightharpoonup\partial_tu_{\infty}$ weakly in $L^2(Q_1^-)$ and $u_i^{-p}\to u_{\infty}^{-p}$ strongly in $L^1(Q_1^-)$;
			\item [(iii)] outside $\{u=0\}$, $u_i$ converges to $u_{\infty}$  locally smoothly;
			\item [(iv)] $u_\infty$ is a suitable weak solution of \eqref{eqn} in $Q_1^-$.
		\end{description}
	\end{thm}
	
	Finally, we use this convergence theorem and blow up analysis to investigate the structure of the rupture set $\{u=0\}$.
	\begin{thm}\label{thm 1.7}
		If $u$ is a $C^{\alpha,\alpha/2}$-continuous suitable weak solution of \eqref{eqn} in $Q_1$, then the parabolic Hausdorff dimension of $\{u=0\}$  is at most $n$.
	\end{thm}
	A more precise result about the stratification of  this rupture set  and its time slices will be proved in Section \ref{section stratification}, see Theorem \ref{thm stratification} and Theorem \ref{thm stratificatoin for slice}.

	\section{Preliminary}\label{section 2}
	\setcounter{equation}{0}
	
	\subsection{Some notations}
	
	Let us first introduce some notations and definitions about the Hausdorff measure with respect to both the Euclidean and parabolic metrics. Throughout this paper, we will use upper case letters $X=(x,t), Y=(y,s), Z=(z,\tau)$,$\dots$ to denote points in space-time and lower case letters $x, y, z$,$\dots$ to denote points in $\R^n$. The parabolic cylinder and backward parabolic cylinder are defined as usual,
	\[Q_{r}(X):=B_{r}(x)\times(t-r^2,t+r^2)\quad \text{and} \quad Q_{r}^-(X):=B_{r}(x)\times(t-r^2,t].\]
	If the center is the origin $(0,0)$, it will be omitted.
	
	Recall that the parabolic distance is defined by
	\[\delta(X,Y):=\max\left\{|x-y|, |t-s|^{1/2}\right\}.\]
	For a set $A\subset\R^n$, $x_0\in\R^n$ and $\lambda>0$, we let $A_{x_0,\lambda}:={(A-x_0)}/{\lambda}.$ For $X,X_0\in\R^{n+1}$ and $E\subset\R^{n+1}$, we define the parabolic scaling $\mathcal{D}$ to be
	\[\mathcal{D}_{X_0,\lambda}(X):=\left(\frac{x-x_0}{\lambda}, \frac{t-t_0}{\lambda^2}\right) \quad \text{and} \quad E_{X_0,\lambda}:=\left\{\mathcal{D}_{X_0,\lambda}(X): X\in E\right\}.\]

	Next we recall the definitions of parabolic Hausdorff measure and parabolic Hausdorff dimension.
	\begin{defi}
		For any $s\geq 0$ and $E\subset\R^{n+1}$, the $s$-dimensional parabolic Hausdorff measure of $E$ is defined by
		\[
		\mathcal{P}^s(E):=\lim\limits_{\delta\to0}\mathcal{P}^s_\delta(E),\]
  where
  \[\mathcal{P}^s_\delta(E):=\inf\left\{\sum_{i}r_i^s: E\subset\cup_{i}Q_{r_i}(X_i),~X_i\in E,~r_i\leq\delta\right\}.
		\]
	\end{defi}
	\begin{defi}
		The parabolic Hausdorff dimension of a set $E\subset\R^{n+1}$ is the number
		\[
		dim_{\mathcal{P}}(E):=\sup\left\{s\geq0: \mathcal{P}^s(E)=+\infty\right\}=\inf\left\{s\geq0: \mathcal{P}^s(E)=0\right\}.
		\]
	\end{defi}
	To avoid abusing of notations, we use $\mathcal{H}^s$ to denote the usual $s$-dimensional Hausdorff measure in $\R^n$ and $dim_{\mathcal{H}}$ to denote the corresponding Hausdorff dimension.

\subsection{Local well-posedness}
	
Here we clarify the local existence of  classical solutions to the Cauchy-Dirichlet problem. We use the assumptions on $\varphi$ in Section \ref{sec introduction}, i.e. $\varphi\in C^{\alpha,\alpha/2}(\partial^p\Omega_T)$ and $\varphi>0$.
\begin{prop}
There exists $T_1>0$ and $u\in C^\infty(\Omega_{T_1})\cap C^{\alpha,\alpha/2}(\overline{\Omega_{T_1}})$ solving the Cauchy-Dirichlet problem for \eqref{eqn} in $\Omega_{T_1}$.
\end{prop}
\begin{proof}
Let $\varepsilon>0$ such that $\varphi\geq 2\varepsilon$. Set
	\[
	f_\varepsilon(u):=
	\begin{cases}
	u^{-p}, & \text{if} ~~ u>\varepsilon,\\
	\varepsilon^{-p-1}u, & \text{otherwise},
	\end{cases}
		\]
which is a Lipschitz function of $u$. By standard parabolic theory, there exists a solution $u_\varepsilon$ to the Cauchy-Dirichlet problem
	\[
	\left\{\begin{aligned}
		& \partial_t u_\varepsilon-\Delta u_\varepsilon=-f_\varepsilon(u_\varepsilon) \quad \text{in}~~\Omega_\infty:=\Omega\times(0,+\infty),\\
		& u_\varepsilon=\varphi \quad \text{on}~~\partial^p\Omega_\infty.
	\end{aligned}\right.
		\]
Because $f_\varepsilon(0)=0$ and $f_\varepsilon$ is a Lipschitz function, the strong maximum principle implies that $u_\varepsilon>0$ in $\Omega_\infty$.
Furthermore, $u_\varepsilon\in C^{\alpha,\alpha/2}(\overline{\Omega_\infty})$. Thus there exists a $T_1>0$ such that
$u_\varepsilon>\varepsilon$ in $\overline{\Omega_{T_1}}$. Then $u_\varepsilon$ solves \eqref{eqn} in $\Omega_{T_1}$. Its smoothness in $\Omega_{T_1}$ follows from standard parabolic regularity theory.
\end{proof}	
This local solution is strictly positive, so it is a classical solution. It can be continued further until it touches $0$, that is, until the first quenching moment.

	\section{H\"{o}lder estimate}\label{section 3}
	\setcounter{equation}{0}

	In this section, we  prove Theorem \ref{thm 1.1} and Theorem \ref{thm 1.3} by using the blow up method and the following Liouville theorem on two-valued, stationary caloric functions.

	\begin{thm}\label{thm Liouville}
		Let $u$ be a stationary two-valued caloric function in $\R^n\times (-\infty,0]$. If $u$ is globally $C^{\gamma,\gamma/2}$-continuous for some $\gamma\in(0,1)$, then $u$ must be a constant.
	\end{thm}
	This theorem is implicitly contained in \cite[Section 5]{Dancer-Wang-Zhang}. For completeness, a proof  will be given in Appendix \ref{appendix Liouville}.
	
	\subsection{Proof of Theorem \ref{thm 1.1}}
	In this subsection we prove Theorem \ref{thm 1.1}. In fact, we will prove a stronger result, which allows varying initial-boundary values.
	\begin{thm}\label{thm Holder estimate, general}
		Suppose $T_i$ is a sequence of positive constants tending to a positive limit $T$, $\varphi_i$ is a sequence of positive $\alpha$-H\"{o}lder continuous functions on $\partial^p\Omega_{T_i}$ with uniformly bounded $\alpha$-H\"{o}lder norm, and $u_i \in C^{\alpha,\alpha/2}(\overline{\Omega_{T_i}})$ is a sequence of positive  solutions to \eqref{eqn} in $\Omega_{T_i}$ with initial-boundary value $\varphi_i$. 
		Then
		\[\sup_i \|u_i\|_{C^{\alpha,\alpha/2}\left(\overline{\Omega_{T_i}}\right)}<+\infty.\]
	\end{thm}
	In the setting of Theorem \ref{thm 1.1}, by our definition of the quenching time $T$, for any $t\in(0,T)$, $u\in C^{\alpha,\alpha/2}(\overline{\Omega_t})$. This theorem then gives
	\[\limsup_{t\to T}\|u\|_{C^{\alpha,\alpha/2}(\overline{\Omega_t})}<+\infty,\]
	which means $u\in C^{\alpha,\alpha/2}\left(\overline{\Omega_T}\right)$.
	
	To prove Theorem \ref{thm Holder estimate, general}, let us assume by the contrary that 
	\[\sup_{X\neq Y\in\Omega_{T_i}} \frac{|u_i(X)-u_i(Y) |}{\delta(X,Y)^{\alpha}}\to +\infty \quad \text{as} ~~ i\to\infty.\]
Then	 there exist $X_i=(x_i,t_i), Y_i=(y_i,s_i)\in\overline{\Omega_{T_i}}$ such that
	\begin{equation}\label{31}
		L_i:= \frac{|u_i(X_i)-u_i(Y_i) |}{\delta(X_i,Y_i)^{\alpha}}\geq \frac{1}{2} \sup_{X\neq Y\in\Omega_{T_i}} \frac{|u_i(X)-u_i(Y) |}{\delta(X,Y)^{\alpha}},
	\end{equation}which tends to $+\infty$ as $i\to\infty$.
	
	Let $\lambda_i=\delta(X_i,Y_i)$ and $Z_i=(\frac{y_i-x_i}{\lambda_i},\frac{s_i-t_i}{\lambda_i^2})$. Define the blow-up sequence 
	\[
	\widetilde{u}_i(X):=L_i^{-1}\lambda_i^{-\alpha }u_i(x_i+\lambda_ix,t_i+\lambda_i^2 t),
	\]
	for $X=(x,t)\in \Omega^i:=\mathcal{D}_{X_i,\lambda_i}(\Omega_{T_i})$.

	A direct calculation shows that $\widetilde{u}_i$ satisfies
	\begin{equation}\label{e1}
		\partial_t\widetilde{u}_i-\Delta \widetilde{u}_i=-\varepsilon_i\widetilde{u}_i^{-p}~\text{in}~\Omega^i,
	\end{equation}
	where $\varepsilon_i=L_i^{-(p+1)}\to 0$. Moreover, a scaling of \eqref{31} yields
	\begin{equation}\label{p2}
		\begin{split}
			|\widetilde{u}_i(0)-\widetilde{u}_i(Z_i)|=1 \quad  \text{and} \quad 
			\sup_{\overline{\Omega^i}} \frac{|\widetilde{u}_i(X)-\widetilde{u}_i(Y)|}{\delta(X,Y)^{\alpha}}\leq 2.
		\end{split}
	\end{equation}
	Because $\delta(0,Z_i)=1$, we may assume $Z_i\to Z_\infty$. Clearly, we still have $\delta(0,Z_\infty)=1$.
	
	Since $u_i=\varphi_i$ on $\partial^p\Omega_{T_i}$, $\widetilde{u}_i=\widetilde{\varphi}_i$ on $\partial^p\Omega^i$, where
	\[\widetilde{\varphi}_i(x,t):=L_i^{-1}\lambda_i^{-\alpha }\varphi_i(x_i+\lambda_ix,t_i+\lambda_i^2 t).\]

	{\bf Case 1.} $\lambda_i\to0$.
	
	As $i\to\infty$, $\Omega^i$ converges to a limit domain, $\Omega^\infty$, which could be
	$\R^{n+1}$, $\R^n\times (-\infty, T_0]$, $H_\infty\times \R$, $H_\infty\times (-\infty, T_0]$, $\R^n\times[T_0,+\infty)$ or $H_\infty\times(T_0,+\infty)$,
	where $H_\infty$ is an half space in $\R^n$ and $T_0\in \R$ is a constant that may be different from one possibility to another.
	Next we divide our analysis further into  two subcases, depending on the behavior of $A_i:=\widetilde{u}_i(0)$.
	
	{\bf Subcase 1.1.} $A_i\to +\infty$.
	
	Thanks to \eqref{p2}, perhaps after passing to a subsequence, we may assume $\widetilde{u}_i-A_i\to \widetilde{u}_\infty$ uniformly in any compact subset of $\Omega^\infty$. Then sending $i\to\infty$ in \eqref{p2} leads to
	\begin{equation}\label{normalization 1}
		1=|\widetilde{u}_{\infty}(0)-\widetilde{u}_{\infty}(Z_{\infty})|\geq \frac{1}{2}\sup_{\Omega^\infty}\frac{|\widetilde{u}_{\infty}(X)-\widetilde{u}_{\infty}(Y) |}{\delta(X,Y)^\alpha},
	\end{equation}
	which implies $\widetilde{u}_\infty$ is globally H\"{o}lder continuous in $\Omega^\infty$.
	
	Since $\varphi_i\in C^{\alpha,\alpha/2}(\partial^p\Omega_{T_i})$, for any $X=(x,t), Y=(y,s)\in\partial^p{\Omega^i}$,
	\[
	\begin{split}
		|\widetilde{\varphi}_i(X)-\widetilde{\varphi}_i(Y)|&\leq L_i^{-1}\lambda_i^{-\alpha}\left|\varphi_i(X_i+\mathcal{D}_{\lambda_i^{-1}}(X))-\varphi_i(X_i+\mathcal{D}_{\lambda_i^{-1}}(Y))\right|\\
        &\leq 2L_i^{-1}\delta(X,Y)^\alpha.
	\end{split}
	\]
	By the uniform convergence of $u_i$, we see, when $\partial^p\Omega^\infty$ is not empty, there exists a constant $A$ such that $\widetilde{u}_\infty\equiv A$ on $\partial_pQ^\infty$.

	By \eqref{p2}, for any $R>0$, if $i$ is large enough, then
	\[
	\begin{split}
		\inf_{Q_R\cap \Omega^i}\widetilde{u}_i\geq A_i - R^{\alpha}\geq \frac{A_i}{2}.
	\end{split}
	\]
	Plugging this lower bound into \eqref{e1}, we get
	\[
	0 \geq (\partial_t-\Delta)(\widetilde{u}_i-A_i) \geq -2^p \varepsilon_i A_i^{-p} \to 0 \quad \text{as $i\to\infty$}.
	\]
	Hence for any $q>1$, by standard $W^{2,1}_q$ estimates for heat equation, $(\widetilde{u}_i-A_i)$ are uniformly bounded in $W^{2,1}_{q,loc}(\Omega^\infty)$ . Then by Sobolev embedding theorem, for any $\gamma\in(0,1)$,  $(\widetilde{u}_i-A_i)$ are uniformly bounded in $C^{1+\gamma,(1+\gamma)/2}_{loc}(\Omega^\infty)$. Letting $i\to\infty$ in \eqref{e1}, we get
	\[
	\partial_t \widetilde{u}_{\infty}-\Delta \widetilde{u}_{\infty}=0~~\text{in}~~\Omega^\infty.
	\]
	Since $\widetilde{u}_\infty$ is globally H\"{o}lder continuous, if $\Omega^\infty=\R^{n+1}$ or $\R^n\times(-\infty,T_0)$, by the generalized Liouville theorem for heat equation, $\widetilde{u}_\infty$ is a constant function. This is a contradiction with the first equality in  \eqref{normalization 1}.
	
	If $\Omega^\infty=H_\infty\times (-\infty,T_0]$, where $H_\infty$ is an half space,  we have shown that $\widetilde{u}_\infty\equiv A$ on $\partial H_\infty\times (-\infty,T_0)$. Extend $\widetilde{u}_\infty$ to  $\R^n\times (-\infty,T_0]$ by taking the odd extension of $\widetilde{u}_\infty-A$ outside $\Omega_\infty$. Then $\widetilde{u}_\infty$ is still a globally H\"{o}lder continuous solution to the heat equation in  $\R^n\times(-\infty,T_0]$, and we arrive at a contradiction as before.
	The case when $\Omega^\infty=H_\infty\times\R$ can be proved in the same way.
	
	If $\Omega^\infty=\R^n\times[T_0,+\infty)$, because $\widetilde{u}_\infty$ is globally H\"{o}lder continuous and $\widetilde{u}_\infty\equiv A$ on $\R^n\times\{T_0\}$, the uniqueness of solutions to the Cauchy problem for heat equation implies that $\widetilde{u}_\infty\equiv A$ in $\Omega^\infty$. If $\Omega^\infty=H_\infty\times[T_0,+\infty)$, we get the same conclusion by using the uniqueness of solutions to the Cauchy-Dirichlet problem for heat equation, because in this case $\widetilde{u}_\infty\equiv A$ on $\partial^p\Omega^\infty$. Both conclusions lead to a contradiction as before.

	{\bf Subcase 1.2.} $A_i\in [0,+\infty)$.

	Since $(\widetilde{u}_i)$ are locally uniformly bounded and uniformly H\"{o}lder continuous, after passing to a subsequence, we may assume $\widetilde{u}_i\to \widetilde{u}_\infty$ locally uniformly in $\Omega^\infty$.
	
	By \eqref{p2}, for large $i$, we have
	\begin{equation}\label{p3}
		1\leq \widetilde{u}_i(0)+\widetilde{u}_i(Z_i).
	\end{equation}
	Letting $i\to\infty$ gives us
	\[
	1\leq \widetilde{u}_\infty(0)+\widetilde{u}_\infty(Z_\infty).
	\]
	Thus $D:=\{\widetilde{u}_\infty>0\}$ is non-empty.  As in Case 1, we have
	\begin{equation}\label{caloric in positivity set}
		\partial_t \widetilde{u}_\infty-\Delta \widetilde{u}_\infty=0~~\text{in}~~D.
	\end{equation}
	If $\{\widetilde{u}_\infty=0\}=\emptyset$, then $\widetilde{u}_\infty$ is caloric in the whole $\Omega^\infty$.
	As in Case 1, we get a contradiction.

	Now we assume $\{\widetilde{u}_\infty=0\}\neq\emptyset$. We want to show that 
 \begin{lem}\label{lem 4.03}
 $\widetilde{u}_\infty$ is a two-valued stationary caloric function in $\Omega^\infty$.
 \end{lem}
To prove this lemma, we need the following convergence results.
	\begin{lem}\label{lem 4.01}
		There hold
		\begin{equation}\label{G1}
			\left\{\begin{aligned}
				& \partial_t \widetilde{u}_i\rightharpoonup\partial_t\widetilde{u}_{\infty}  \quad \text{weakly in}~~L^2_{loc}(\Omega^\infty),\\
				&\nabla \widetilde{u}_i\to\nabla \widetilde{u}_{\infty} \quad \text{strongly in}~~L^2_{loc}(\Omega^\infty),\\
				&\varepsilon_i\widetilde{u}_i^{1-p}\to0 \quad \text{in}~~L^1_{loc}(\Omega^\infty).
			\end{aligned}\right.
		\end{equation}
	\end{lem}
	\begin{proof}
		Take an $Q_R(X_0)\subset\Omega^\infty$ and an arbitrary $\eta\in C_0^{\infty}(Q_R(X_0))$. By our assumption, for any $i$, $\widetilde{u}_i>0$ in $Q_R(X_0)$ and thus $\widetilde{u}_i$ is smooth in $Q_R(X_0)$. Testing \eqref{e1} with $\widetilde{u}_i\eta^2$, we  obtain
		\begin{equation}\label{G3}
			-\int_{Q_R(X_0)} \widetilde{u}_i^2\eta\partial_t\eta+\int_{Q_R(X_0)} |\nabla \widetilde{u}_i|^2\eta^2+\int_{Q_R(X_0)} \varepsilon_i\widetilde{u}_i^{1-p}\eta^2=-\int_{Q_R(X_0)} 2\eta \widetilde{u}_i\nabla \widetilde{u}_i\cdot\nabla\eta.
		\end{equation}
		
		It follows from the Cauchy-Schwarz inequality that
		\begin{equation}\label{G4}
			\int_{Q_R(X_0)} |\nabla \widetilde{u}_i|^2\eta^2+\int_{Q_R(X_0)} \varepsilon_i \widetilde{u}_i^{1-p}\eta^2\leq C\left(\int_{Q_R(X_0)} \widetilde{u}_i^2|\nabla\eta|^2+\int_{Q_R(X_0)} \widetilde{u}_i^2\eta\partial_t\eta\right).
		\end{equation}
		Testing \eqref{e1} with $\partial_t\widetilde{u}_i\eta^2$, by the Cauchy-Schwarz inequality (or equivalently, using the localized energy identity for $\widetilde{u}_i$), we get
		\begin{equation}\label{G5}
			\begin{split}
				\int_{Q_R(X_0)}(\partial_t\widetilde{u}_i)^2\eta^2&\leq C\int_{Q_R(X_0)} \left(|\nabla \widetilde{u}_i|^2+\varepsilon_i\widetilde{u}_i^{1-p}\right)\eta\partial_t\eta.
			\end{split}
		\end{equation}
		Since $(\widetilde{u}_i)$ are  uniformly bounded in any compact set of $\Omega^\infty$, \eqref{G4} and \eqref{G5} imply that $\partial_t\widetilde{u}_i$ and $\nabla \widetilde{u}_i$ are uniformly bounded in $L^2_{loc}(\Omega^\infty)$. Thus up to subsequence, $\partial_t\widetilde{u}_i\rightharpoonup\partial_t\widetilde{u}_{\infty}$ and $\nabla \widetilde{u}_i\rightharpoonup\nabla \widetilde{u}_{\infty}$ weakly in $L^2_{loc}(\Omega^\infty)$.
		
		Sending $i\to\infty$ in \eqref{G3} we also get
		\[
		\begin{split}
			&\lim\limits_{i\to\infty}\int_{Q_R(X_0)}\left(|\nabla \widetilde{u}_i|^2-|\nabla \widetilde{u}_{\infty}|^2\right)\eta^2+\lim\limits_{i\to\infty}\int_{Q_R(X_0)} \varepsilon_i\widetilde{u}_i^{1-p}\eta^2\\
			=&-\int_{Q_R(X_0)} |\nabla \widetilde{u}_{\infty}|^2\eta^2+2\eta \widetilde{u}_{\infty}\nabla \widetilde{u}_{\infty}\cdot\nabla\eta+\partial_t \widetilde{u}_{\infty} \widetilde{u}_{\infty}\eta^2.
		\end{split}
		\]
		For any $t$, by Sard theorem, we can take a  small $\tau>0$ so that $\{\widetilde{u}_{\infty}=\tau\}\cap\{t\}$ is a smooth hypersurface. Because $\partial_t \widetilde{u}_{\infty}=\Delta \widetilde{u}_{\infty}$ in $\{\widetilde{u}_{\infty}>\tau\}$, we have
		\begin{eqnarray*}
			&&\int_{\{\widetilde{u}_{\infty}>\tau\}\cap\{t\}}\left[|\nabla \widetilde{u}_{\infty}|^2\eta^2+2\eta \widetilde{u}_{\infty}\nabla \widetilde{u}_{\infty}\cdot\nabla\eta+\partial_t \widetilde{u}_{\infty}\widetilde{u}_{\infty}\eta^2\right]dx\\
			&=&\tau\int_{\{\widetilde{u}_{\infty}=\tau\}\cap\{t\}}\frac{\partial \widetilde{u}_{\infty}}{\partial\nu_t}\eta^2 dx\\
			&=&\tau\int_{\{\widetilde{u}_{\infty}>\tau\}\cap\{t\}}\left(\nabla \widetilde{u}_{\infty}\cdot\nabla\eta^2+\partial_t\widetilde{u}_\infty \eta^2\right)dx\\
			&=&O\left(\tau \left(\int_{\{\widetilde{u}_{\infty}>\tau\}\cap\{t\}}|\nabla \widetilde{u}_\infty|^2+|\partial_t\widetilde{u}_\infty|^2 dx\right)^{1/2}\right),
		\end{eqnarray*}
		where $\nu_t$ is the unit outer normal vector to $\{\widetilde{u}_{\infty}=\tau\}\cap\{t\}$. Letting $\tau\to0$ and then integrating in $t$, we see
		\[-\int_{Q_R(X_0)} |\nabla \widetilde{u}_{\infty}|^2\eta^2+2\eta \widetilde{u}_{\infty}\nabla \widetilde{u}_{\infty}\cdot\nabla\eta+\partial_t \widetilde{u}_{\infty} \widetilde{u}_{\infty}\eta^2=0.\]
		Hence
		\[
		\lim\limits_{i\to\infty}\int_{Q_R(X_0)}\left(|\nabla \widetilde{u}_i|^2-|\nabla \widetilde{u}_{\infty}|^2\right)\eta^2+\lim\limits_{i\to\infty}\int_{Q_R(X_0)} \varepsilon_i\widetilde{u}_i^{1-p}\eta^2=0.
		\]
		By the weak lower semi-continuity of $L^2-$norm, this establishes \eqref{G1}.
	\end{proof}

\begin{proof}[Proof of Lemma \ref{lem 4.03}]
    Because $\widetilde{u}_\infty\geq 0$, combining \eqref{caloric in positivity set} with the Kato inequality, we see $\widetilde{u}_\infty$ is subcaloric.
	In fact, by \eqref{caloric in positivity set}, $\Delta \widetilde{u}_\infty-\partial_t\widetilde{u}_\infty$ is a positive Radon measure supported in $\{\widetilde{u}_\infty=0\}$. From this property, and because
	$\widetilde{u}_\infty$ is continuous, we deduce that
	\[\widetilde{u}_\infty(\partial_t\widetilde{u}_\infty-\Delta \widetilde{u}_\infty)=0\]
	in the weak sense, that is, $\widetilde{u}_\infty$ satisfies (iii) in Definition \ref{stationary two-valued caloric fct}.

	Next we use Lemma \ref{lem 4.01} to show that $\widetilde{u}_\infty$ satisfies \eqref{stationary condition 2} and \eqref{localized energy inequality 2}. Take an arbitrary $Y\in C_0^{\infty}(\Omega_\infty,\R^n)$. Multiplying \eqref{e1} by $Y\cdot\nabla \widetilde{u}_i$ and integrating by parts give
	\[
	\int \left(\frac{1}{2}|\nabla \widetilde{u}_i|^2-\frac{\varepsilon_i\widetilde{u}_i^{1-p}}{p-1}\right)\text{div} Y-DY(\nabla \widetilde{u}_i,\nabla \widetilde{u}_i)-\partial_t\widetilde{u}_i(\nabla \widetilde{u}_i\cdot Y)=0.
	\]
	By sending $i\to\infty$ and using Lemma \ref{lem 4.01}, we get
	\[
	\int \frac{1}{2}|\nabla \widetilde{u}_\infty|^2\text{div} Y-DY(\nabla \widetilde{u}_\infty,\nabla \widetilde{u}_\infty)-\partial_t\widetilde{u}_\infty(\nabla \widetilde{u}_\infty\cdot Y)=0,
	\]
	that is, $\widetilde{u}_{\infty}$ is stationary.
	
	To show $\widetilde{u}_\infty$ satisfies \eqref{localized energy inequality}, for any $\eta\in C_0^\infty (\Omega_\infty)$, testing \eqref{e1} with $\partial_t \widetilde{u}_i\eta^2$ and integrating by parts, we obtain, for any $t_1<t_2$,
	\begin{equation}\label{G6}
		\begin{split}
			&\int_{\R^{n}}\left(\frac{|\nabla \widetilde{u}_i|^2}{2}-\frac{\varepsilon_i}{p-1}\widetilde{u}_i^{1-p}\right)\eta^2 \Bigg\lfloor_{t_1}^{t_2}\\
			=&-\int_{t_1}^{t_2}\int_{\R^{n}}|\partial_t\widetilde{u}_i|^2 \eta^2-2\int_{t_1}^{t_2}\int_{\R^{n}}\partial_t\widetilde{u}_i\eta \nabla \widetilde{u}_i \cdot\nabla\eta\\
			&+2\int_{t_1}^{t_2}\int_{\R^{n}}\left(\frac{|\nabla \widetilde{u}_i|^2}{2}-\frac{\varepsilon_i}{p-1}\widetilde{u}_i^{1-p}\right)\eta \partial_t\eta.
		\end{split}
	\end{equation}
	Passing to the limit in \eqref{G6} and using the convergence results in Lemma \ref{lem 4.01}, we deduce that $\widetilde{u}_\infty$ satisfies the local energy inequality \eqref{localized energy inequality 2}. 
	\end{proof}
 
	If $\Omega^\infty=\R^{n+1}$ or $\R^n\times(-\infty,T_0]$, then Theorem \ref{thm Liouville} is applicable, which implies that $\widetilde{u}_{\infty}$ is a constant. This is a contradiction since we have assumed that both $\{\widetilde{u}_{\infty}=0\}$ and $\{\widetilde{u}_{\infty}>0\}$ are  nonempty.
	
	If $\Omega^\infty=H_\infty \times (-\infty, T_0]$ or $H_\infty \times \R$, $\widetilde{u}_\infty$ is still a globally H\"{o}lder continuous two-valued caloric function in $\Omega^\infty$.  Moreover, we have
	\begin{lem}\label{lem 4.02}
		There exists a constant $A\geq 0$ such that $\widetilde{u}_\infty\equiv A$ on $\partial H_\infty$, and $0\leq \widetilde{u}_\infty \leq A$ in $\Omega^\infty$.
	\end{lem}
	If this claim is true, we can apply Theorem \ref{thm half space} to deduce that $\widetilde{u}_{\infty}\equiv A$. This is a contradiction as before.

	\begin{proof}[Proof of Lemma \ref{lem 4.02}]
Let $\Phi_i$ be the solution of the heat equation in $\Omega_{T_i}$ with initial-boundary value condition $\Phi_i=\varphi_i$ on $\partial^p\Omega_{T_i}$. By the comparison principle,
		\begin{equation}\label{comparison with heat eqn}
			u_i\leq \Phi_i \quad \text{in} ~~ \Omega_{T_i}.
		\end{equation}
		Similar to $\widetilde{u}_i$ and $\widetilde{\varphi}_i$, define
		\[\widetilde{\Phi}_i(X)=L_i^{-1}\lambda_i^{-\alpha }\Phi_i(x_i+\lambda_ix,t_i+\lambda_i^2 t).\]
		We have the boundary condition $\widetilde{u}_i=\widetilde{\varphi}_i$ on $\partial^p\Omega^i$, and rescaling \eqref{comparison with heat eqn} gives
		\begin{equation}\label{comparison with heat eqn 2}
			\widetilde{u}_i\leq \widetilde{\Phi}_i \quad \text{in} ~~ \Omega^i.
		\end{equation}
		By standard regularity theory for heat equation, $\Phi_i$ are uniformly bounded in $ C^{\alpha,\alpha/2}(\overline{\Omega_{T_i}})$, so there exists a constant $C$ independent of $i$ such that
		\[\|\widetilde{\Phi}_i\|_{C^{\alpha,\alpha/2}(\Omega^i)}\leq CL_i^{-1}.\]
		Then because $\widetilde{\varphi}_i$ converges locally uniformly to $A$,  $\widetilde{\Phi}_i$ also converges to $A$ locally uniformly in $\Omega^\infty$.  Passing to the limit in \eqref{comparison with heat eqn 2} as well as in the initial-boundary condition $\widetilde{u}_i=\widetilde{\varphi}_i$, we get the two conclusions in this lemma.
	\end{proof}

	Finally, if $\Omega^\infty=\R^n\times[T_0,+\infty)$,  $\widetilde{u}_\infty$ is still a globally H\"{o}lder continuous two-valued caloric function in $\Omega^\infty$. Moreover, $\widetilde{u}_\infty\equiv A$ on $\R^n\times\{T_0\}$. If $A=0$, because $\widetilde{u}_\infty$ is sub-caloric, maximum principle implies that  $\widetilde{u}_\infty\equiv  0$ in $\Omega^\infty$. If $A>0$, by the globally H\"{o}lder continuity of $\widetilde{u}_\infty$,  there exists	a $\delta>0$ such that $\widetilde{u}_\infty>0$ in $\R^n\times[T_0,T_0+\delta]$. Then $\widetilde{u}_\infty$ is caloric in $\R^n\times(T_0,T_0+\delta)$. By the uniqueness of solutions to the Cauchy problem for heat equation,  $\widetilde{u}_\infty\equiv A$ in $\R^n\times[T_0,T_0+\delta]$. This procedure can be continued further in time, leading to the conclusion that $\widetilde{u}_\infty\equiv A$ in $\Omega^\infty$. The case when $\Omega^\infty=H_\infty\times[T_0,+\infty)$ can be proved in the same way.

	{\bf Case 2.} $\lambda_i\to\lambda_\infty>0$.
	
	This assumption, together with \eqref{31},  implies that either $u_i(X_i)$ or $u_i(Y_i)\to+\infty$. Without loss of generality, assume $u_i(X_i)\to+\infty$.
	
	Recall that $\Phi_i$ is the solution of  heat equation in $\Omega_{T_i}$ with $\Phi_i=\varphi_i$ on $\partial^p\Omega_{T_i}$.
	Since $\partial_tu_i-\Delta u_i\leq0$ and it has the same initial-boundary value, comparison principle implies that  $0\leq u_i\leq \Phi_i$ in $\Omega_{T_i}$. Hence $\Phi_i(X_i)\to+\infty$. Because $\Phi_i$ and $\varphi_i$ are uniformly $\alpha$-H\"{o}lder continuous (on $\Omega_{T_i}$ and $\partial^p\Omega_{T_i}$, respectively), and $0\leq t_i\leq T_i$ are bounded, we must have
	\[\text{dist}(x_i,\partial\Omega)\to+\infty.\]
	Together with the conditions on $\lambda_i$ and $T_i$, this implies that $\Omega^i$ converges to  $\Omega^\infty=\R^n\times(a,b)$ for two constants $a\leq 0 <b$. Then we can get a contradiction as before, by using the uniqueness of the solution to the Cauchy problem for heat equation (when $\widetilde{u}_i(0)\to+\infty$) or by a continuation argument (when $\widetilde{u}_i(0)\to A\in[0,+\infty)$).

	{\bf Case 3.} $\lambda_i\to+\infty$.
	
	By the uniform H\"{o}lder continuity of $\varphi_i$, we have
	\begin{eqnarray*}
	L_i \lambda_i^\alpha	&=&|u_i(x_i,t_i)-u_i(y_i,s_i)|\\
		&\leq &|u_i(x_i,t_i)-\varphi_i(x_i,0)|+|\varphi_i(x_i,0)-\varphi_i(y_i,0)|+|\varphi_i(y_i,0)-u_i(y_i,s_i)| \\
		&\leq & |u_i(x_i,t_i)-\varphi_i(x_i,0)|+C|x_i-y_i|^\alpha+|\varphi_i(y_i,0)-u_i(y_i,s_i)|\\
		&\leq & |u_i(x_i,t_i)-\varphi_i(x_i,0)|+C\lambda_i^\alpha+|\varphi_i(y_i,0)-u_i(y_i,s_i)|
	\end{eqnarray*}
Hence either $|u_i(x_i,t_i)-\varphi_i(x_i,0)|$ or $|\varphi_i(y_i,0)-u_i(y_i,s_i)|$  is larger than $L_i\lambda_i^\alpha/4$. Assume it is the first case.
Then because $0\leq t_i\leq T_i\leq 2T$ for all $i$ large, we have
\[\frac{|u_i(x_i,t_i)-\varphi_i(x_i,0)|}{t_i^{\alpha/2}}\geq \frac{L_i\lambda_i^\alpha}{16T^{\alpha/2}} >2L_i.\]
This is a contradiction with the assumption that the H\"{o}lder semi-norm of $u_i$ is not larger than $2L_i$. In conclusion, this case cannot appear if $i$ is large enough.
	
	\subsection{Proof of Theorem \ref{thm 1.3}}	
	In this subsection, we use similar ideas to prove Theorem \ref{thm 1.3}, which is about interior H\"{o}lder estimate. So we need to take a localization, but we do not need to consider the near boundary case in the previous subsection.	
	\begin{proof}[{Proof of Theorem \ref{thm 1.3}}]
		Take an $\eta\in C_0^\infty(Q_2^-)$ satisfying
		\[
		\eta=\left\{
		\begin{array}{lll}
			1 \quad &\text{if}~~(x,t)\in (Q_{1}^-)^\circ\\
			>0 \quad &\text{if}~~(x,t)\in (Q_{3/2}^-)^\circ\\
			0 \quad &\text{otherwise}.
		\end{array}
		\right.
		\]
		Set $\hat{u}_i=u_i\eta$. We only need to prove
		\[\sup _i \left \|\hat{u}_i\right \|_{C^{\alpha,\alpha/2}(\overline{Q_{3/2}^-})}<+\infty.\]
		Assume by the contrary that there exist $X_i=(x_i,t_i), Y_i=(y_i,s_i)\in \overline{Q_{3/2}^-}$ such that
		\[
		L_i:=\frac{|\hat{u}_i(X_i)-\hat{u}_i(Y_i)|}{\delta(X_i,Y_i)^\alpha}=\sup_{\overline{Q_{3/2}^-}} \frac{|\hat{u}_i(X)-\hat{u}_i(Y)|}{\delta(X,Y)^\alpha}\to+\infty.
		\]
		Since $\partial_tu_i-\Delta u_i\leq0$ and $\sup_{i}\int_{Q_2^-}u_i(x,t)<+\infty$,  $\sup_{i}\|u_i\|_{L^\infty(Q_{5/3}^-)}<+\infty$ and $\sup_{i}\|\hat{u}_i\|_{L^\infty(Q_{3/2}^-)}<+\infty$. Thus $\delta(X_i,Y_i)\to0$ as $i\to\infty$.
		
		Let $\lambda_i=\delta(X_i,Y_i)$, $Z_i=\mathcal{D}_{X_i,\lambda_i}(Y_i)$. Define two blow-up sequences
		\[
		{\widetilde{u}}_i(x,t):=L_i^{-1}\lambda_i^{-\alpha }\widehat{u}_i(x_i+\lambda_ix,t_i+\lambda_i^2 t).
		\]
		and
		\[
		\bar{u}_i(x,t):= L_i^{-1}\lambda_i^{-\alpha}u_i(x_i+\lambda_ix,t_i+\lambda_i^2 t)\eta(x_i,t_i).
		\]
		Here $\widetilde{u}_i$ and $\bar{u}_i$ are defined in $Q^i:=(Q_{3/2}^-)_{X_i,\lambda_i}$. 
  
  By their definitions and a rescaling,
		\begin{equation}\label{223}
			\widetilde{u}_i(x,t)=\frac{\eta(x_i+\lambda_ix,t_i+\lambda_i^2 t)}{\eta(x_i,t_i)}\bar{u}_i(x,t)
		\end{equation}
		and
		\begin{equation}\label{225}
			\begin{split}
				|\widetilde{u}_i(0)-\widetilde{u}_i(Z_i)|&=L_i^{-1}\lambda_i^{-\alpha}|\hat{u}_i(X_i)-\hat{u}_i(Y_i)|\\
				&=\sup_{Q^i}\frac{|\tilde{u}_i(X)-\tilde{u}_i(Y)|}{\delta(X,Y)^\alpha}=1.
			\end{split}
		\end{equation}
		Since $\eta$ is Lipschitz, for any $X\in Q^i$,
		\begin{equation} \label{226}
			\begin{split}
				|\tilde{u}_i(X)-\bar{u}_i(X)| &\leq CL_i^{-1}\lambda_i^{-\alpha}\sup_{Q^i}|\eta(x_i+\lambda_i x,t_i+\lambda_i^2 t)-\eta(x_i,t_i)|
				\\ &\leq C L_i^{-1}\lambda_i^{\frac{p-1}{p+1}} \left(|x|^2+|t|\right)^{1/2}\leq C L_i^{-1}\lambda_i^{\frac{p-1}{p+1}} \delta(X,0),
			\end{split}
		\end{equation}
		where $C>0$ depends only on the Lipschitz constant of $\eta$. Thus $|\tilde{u}_i(x,t)-\bar{u}_i(x,t)|\to0$ locally uniformly in $Q^\infty:=\lim_{i}Q^i$ as $i\to\infty$. Furthermore, by the Lipschitz continuity of $\eta$, we also have
		\begin{equation}\label{227}
			\tilde{u}_i(X) \leq C L_i^{-1}\lambda_i^{-\alpha}\eta(x_i+\lambda_ix,t_i+\lambda_i^2t) \leq C L_i^{-1}\lambda_i^{\frac{p-1}{p+1}}\delta(X,\partial Q^i).
		\end{equation}
		Direct calculation shows that $\bar{u}_i$ satisfies
		\begin{equation}\label{228}
			\partial_t \bar{u}_i - \Delta\bar{u}_i= -\varepsilon_i \bar{u}_i^{-p},
		\end{equation}
		where $\varepsilon_i=L_i^{-(1+p)}\eta(X_i)^{(1+p)}\to0$ as $i\to\infty$.
		
		As in the proof of Theorem \ref{thm 1.1}, we divide the proof into two cases depending on the behavior of $A_i:=\tilde{u}_i(0)$.
		
		{\bf Case 1:} $A_i\to +\infty$.
		
		In this case, by \eqref{227},
		\[
		\delta(0,\partial ^pQ^i) \geq C^{-1}L_i \lambda_i^{-\frac{p-1}{p+1}}A_i \to +\infty \quad \text{as $i\to\infty$},
		\]
		which implies $Q^\infty=\R^n\times(-\infty,0)$ or $\R^{n+1}$. We prove only the first case. Up to subsequence, assume $(\tilde{u}_i-A_i)\to\bar{u}_\infty$ locally uniformly in $\R^n\times(-\infty,0)$. By \eqref{226}, $(\bar{u}_i-A_i)\to\bar{u}_\infty$ locally uniformly in  $\R^n\times(-\infty,0)$. Similar to Case 1 in the proof of Theorem \ref{thm Holder estimate, general}, we  have
		\[
		\partial_t \bar{u}_{\infty}-\Delta\bar{u}_{\infty}=0~~\text{in}~~\R^n\times(-\infty,0),
		\]
		and
		\[
		1=|\bar{u}_{\infty}(0)-\bar{u}_{\infty}(Z_{\infty})|=\sup_{\R^n\times(-\infty,0]}\frac{|\bar{u}_{\infty}(X)- \bar{u}_{\infty}(Y) |}{\delta(X,Y)^\alpha}.
		\]
		Hence $\bar{u}_\infty$ is a globally H\"{o}lder continuous solution to the heat equation in $\R^n\times(-\infty,0]$. By the generalized Liouville theorem for  heat equation, $\bar{u}_\infty$ must be a constant, which is a contradiction.
		
		{\bf Case 2:} $A_i\in [0,+\infty)$.
		
		By \eqref{225} and \eqref{227}, we obtain
		\begin{equation}\label{234}
			1\leq \tilde{u}_i(0)+\tilde{u}_i(Z_i).
		\end{equation}
		and
		\[\label{235}
		C^{-1}L_i \lambda_i^{-\frac{p-1}{p+1}} \leq \delta(0,\partial^p Q^i)+\delta(Z_i,\partial^p Q^i) \leq 2\delta(0,\partial^p Q^i) +1.
		\]
		This means $\delta(0,\partial^p Q^i)\to+\infty$ and thus $Q^\infty=\R^n\times(-\infty,0)$ or $\R^{n+1}$. Since $\tilde{u}_i$ are locally uniformly bounded in $\R^n\times(-\infty,0]$, we may assume $\tilde{u}_i$ and $\bar{u}_i$ converge to $\bar{u}_\infty$ locally uniformly in $\R^{n+1}$. Letting $i\to\infty$ in \eqref{234} gives
		\[\label{236}
		1\leq \tilde{u}_\infty(0)+\tilde{u}_\infty(Z_\infty).
		\]
		where $Z_\infty:=\lim_{i}Z_i$. Thus $D:=\{\bar{u}_\infty>0\}$ is non-empty. If $\{\bar{u}_\infty=0\}=\emptyset$, we argue as in Case 1 to get a contradiction.
		
		If $\{\bar{u}_\infty=0\}\neq\emptyset$, following the analysis in  Case 2 of the previous subsection, we deduce that $\bar{u}_\infty$ is a stationary, two-valued caloric function. Therefore, by Theorem \ref{thm Liouville}, $\bar {u}_{\infty}$ is a constant function, which contradicts the fact that both  $\{\bar{u}_{\infty}=0\}$ and $\{\bar{u}_{\infty}>0\}$ are nonempty.
	\end{proof}

	\section{Convergence theory}\label{section 4}
	\setcounter{equation}{0}

	This section is  devoted to the proof of Theorem \ref{thm 1.6}. Recall that we assume   $(u_i)$ is a sequence of suitable weak solutions of \eqref{eqn} in $Q_2^-=B_2(0)\times (-4,0]$, satisfying
	\begin{equation}\label{uniform Holder assumption}
		\sup\limits_{i}\|u\|_{C^{\alpha,\alpha/2}(\overline{Q_{3/2}^-})}<+\infty.
	\end{equation}
	By Theorem \ref{thm 1.3}, this includes the case that $(u_i)$ are positive solutions satisfying \eqref{14}.
	
	We first present several uniform a priori estimates on $u_i$.
	\begin{lem}\label{lemma 4.1}
		If $X\in\{u_i>0\}\cap Q_1^-$, then $|\nabla u_i(X)|\leq Cu_i(X)^{\frac{1-p}{2}}$.
	\end{lem}
	\begin{proof}
		Let $u_i(X)=:h^{\alpha}>0$. By the uniform H\"{o}lder continuity assumption \eqref{uniform Holder assumption}, there exists a $\sigma>0$ depending only on the H\"{o}lder semi-norm of $u_i$ such that
		\[h^{\alpha}/2 \leq u_i \leq 2h^\alpha \quad \text{in} ~~ Q_{\sigma h}^-(X). \]
		Define
		\[\bar u_i(Y)=h^{-\alpha}u_i(X+\mathcal{D}_{h^{-1}}(Y)).\]
		Then $\bar u_i$ satisfies  \eqref{eqn} and $1/2\leq \bar u \leq 2$ in $Q_\sigma^-(0)$. By standard interior gradient estimates for parabolic equations, $|\nabla \bar u_i(0)|\leq C$. After rescaling back we get the desired estimate.
	\end{proof}
	
	\begin{lem}\label{lemma 4.2}
		There exists a constant $C>0$ independent of $i$, such that
		
		(1)~for any $i$, $X\in Q_1^-$ and $r\in(0,\tfrac{1}{2})$,
		\begin{equation}\label{es1}
			\int_{Q_r^-(X)}u_i^{-p}\leq Cr^{n+\frac{2}{p+1}},
		\end{equation}
		
		(2)~for any $i$, $X\in Q_1^-$ and $r\in(0,\frac{1}{2})$,
		\begin{equation}\label{es2}
			\int_{Q_{r}^-(X)}\left(|\nabla u_i|^2+u_i^{1-p}\right)\leq Cr^{n+\frac{4}{p+1}};
		\end{equation}
		
		(3)~for any $i$, $X\in Q_1^-$ and $r\in(0,\frac{1}{4})$,
		\begin{equation}\label{es3}
			\int_{Q_{r}^-(X)}(\partial_tu_i)^2\leq Cr^{n-2+\frac{4}{p+1}};
		\end{equation}
		
		(4)~ [Nondegeneracy] for any $i$, $X\in Q_1^-$ and $r\in(0,\frac{1}{2})$,
		\begin{equation}\label{es4}
			\int_{Q_{r}^-(X)}u_i\geq Cr^{n+2+\frac{2}{p+1}};
		\end{equation}
		
		(5)~for any $i$, $X, Y\in Q_1^-$,
		\begin{equation}\label{es5}
			|u_i(X)-u_i(Y)|\leq C\delta(X,Y)^{\alpha}.
		\end{equation}
	\end{lem}
	\begin{proof}
 Take $\eta\in C_0^{\infty}\left(Q_{2r}^-(X)\right)$ satisfying
		\[
		\eta\equiv1~\text{in}~(Q_r^-(X))^\circ,~|\nabla \eta|<C r^{-1},~|\Delta\eta|+|\partial_t\eta|\leq C r^{-2}.
		\]
		(1) Testing the equation of $u_i$ with $\eta$, we get
		\[
		\begin{split}
			\int_{Q_r^-(X)}u_i^{-p}&\leq \int u_i^{-p}\eta=\int (\Delta u_i-\partial_t u_i)\eta\\
            &=\int (\Delta\eta+\partial_t\eta)\left[u_i-u_i(X)\right]\leq Cr^{n+\frac{2}{p+1}},
		\end{split}
		\]
		where we have used the $\alpha$-continuity of $u_i$. This gives the estimate in \eqref{es1}.
		
		(2) By H\"{o}lder inequality,
		\[
		\begin{split}
			\int_{Q_r^-(X)}u_i^{1-p}\leq \left(\int_{Q_r^-}u_i^{-p}\right)^{\frac{p-1}{p}}|Q_r^-|^{\frac{1}{p}}\leq Cr^{n+\frac{4}{p+1}}.
		\end{split}
		\]	
		Then using Lemma \ref{lemma 4.1}, we obtain
		\[
		\begin{split}
			\int_{Q_r^-(X)} |\nabla u_i|^2=\int_{Q_r^-(X)\cap\{u>0\}}|\nabla u_i|^2\leq C\int_{Q_r^-(X)} u_i^{1-p}\leq Cr^{n+\frac{4}{p+1}},
		\end{split}
		\]
		which establishes \eqref{es2}.

	(3)	To get inequality \eqref{es3}, we integrate the energy inequality \eqref{localized energy inequality} on $\R$ and use the Cauchy-Schwarz inequality, which leads to
		\[
		\begin{split}
			\int_{Q_{2r}^-(X)} (\partial_tu_i)^2\eta^2&\leq C\left[\int_{Q_{2r}^-(X)}\left(\frac{1}{2}|\nabla u_i|^2-\frac{1}{p-1}u_i^{1-p}\right)2\eta\partial_t\eta+\int_{Q_{2r}^-(X)} |\nabla u_i|^2|\nabla\eta|^2\right]\\
			&\leq Cr^{-2}\int_{Q_{2r}^-(X)}\left[|\nabla u_i|^2+u_i^{1-p}\right]\\
			&\leq C r^{-2}r^{n+\frac{4}{p+1}}=C r^{n-2+\frac{4}{p+1}},
		\end{split}
		\]
		where we have used \eqref{es2} and the fact that $\eta\in C_0^\infty(Q_{2r}^-(X))$ with $r\in(0,1/4)$.

  (4) Testing \eqref{eqn} with $\eta$ gives
  \[\int_{Q_{r/2}^-(X)}u_i^{-p}\leq Cr^{-2} \int_{Q_r^-(X)}u_i.\]
  By the Jensen inequality, we also have
  \[\int_{Q_{r/2}^-(X)}u_i^{-p} \geq \frac{1}{C}r^{-(n+2)(p+1)} \left(\int_{Q_{r/2}^-(X)}u_i\right)^{-p}\geq  \frac{1}{C}r^{-(n+2)(p+1)} \left(\int_{Q_r^-(X)}u_i\right)^{-p}. \]
Combining these two inequalities, we get \eqref{es4}.

	 (5) This is a corollary of the uniform H\"{o}lder continuity of $(u_i)$.
	\end{proof}

	By inequalities \eqref{es2}, \eqref{es3} and \eqref{es5}, we can assume that, up to a subsequence, $u_i\to u_{\infty}$ uniformly in $Q_1^-$, $\partial_tu_i\rightharpoonup\partial_tu_{\infty}$ and $\nabla u_i\rightharpoonup\nabla u_{\infty}$ weakly in $L^2(Q_1^-)$. By this uniform convergence, $u_\infty$ satisfies \eqref{es4} and \eqref{es5}. Furthermore, for any domain $\mathcal{D} \subset\subset \{u_\infty>0\}\cap Q_1^-$ and $k\geq 1$, by standard parabolic estimates and Arzel\`a-Ascoli theorem,  $u_i$ converges to $u_\infty$ in $C^k(\mathcal{D})$.
	
	The remaining proof of Theorem \ref{thm 1.6} is divided  into the following lemmas.
	\begin{lem}\label{lemma 4.3}
		$\mathcal{P}^{n+\alpha}\left(\{u_\infty=0\}\cap Q_1^-\right)=0.$
	\end{lem}
	\begin{proof}
		By inequality (\ref{es4}),
		\[
		\sup_{Q_{r}^-(X)} u_\infty\geq \tfrac{1}{C} r^{\frac{2}{p+1}} \quad \text{for any}~X\in\{u_\infty=0\}\cap Q_1^-~\text{and}~r\in\left(0,\tfrac{1}{2}\right).
		\]
		By the H\"{o}lder continuity of $u_\infty$, there exists a $Q_{\sigma r}^-(Y)\subset Q_r^-(X)$, where $\sigma$ depends only on the H\"{o}lder semi-norm of $u_\infty$, such that
		\[
		u_\infty\geq C r^{\frac{2}{p+1}} \quad \text{in}~Q_{\sigma r}^-(Y).
		\]
		This means for any $X\in \{u_\infty=0\}\cap Q_1^-$ and $r\in(0,\frac{1}{2})$, we can find
		\[
		Q_{\sigma r}^-(Y)\subset\{u_\infty>0\}\cap Q_1^-.
		\]
		As a consequence, the $\mathcal{P}^{n+2}$-density of $\{u_\infty=0\}$ at $X$ is strictly less than $1$. Since $\mathcal{P}^{n+2}$ is equivalent to the standard Lebesgue measure on $\R^{n+1}$, by the Lebesgue differentiation theorem, we deduce that $\mathcal{P}^{n+2}(\{u_\infty=0\}\cap Q_1^-)=0$.
		
		In any compact subset of $\{u_\infty>0\}\cap Q_1^-$, $u_i^{-p}$ converges to $u_\infty^{-p}$ uniformly, hence $u_i^{-p}\to u_\infty^{-p}$,  $\mathcal{P}^{n+2}$ $a.e.$ in $Q_1^-$. By Fatou's lemma,
		\[
		\int_{Q_1^-} u_\infty^{-p}\leq \liminf\limits_{i\to\infty}\int_{Q_1^-} u_i^{-p}\leq C.
		\]
		
		For any $\varepsilon>0$, take a maximal $\varepsilon$-separated set
		\[
		\{X_i\}_{i=1}^N\subset \{u_\infty=0\}\cap Q_1^-.
		\]
		Then $\{Q_{\varepsilon/2}(X_i)\}$ are disjoint. Moreover, by denoting $\mathcal{N}_\varepsilon$ to be the $\varepsilon$-neighborhood of $\{u_\infty=0\}$ in $Q_1$, we have $\{u_\infty=0\}\subset \cup_{i=1}^N Q_{\varepsilon}(X_i)\subset\mathcal{N}_\varepsilon$.
		Because
		\[\lim_{\varepsilon\to 0}\mathcal{P}^{n+2}(\mathcal{N}_\varepsilon)=\mathcal{P}^{n+2}(\{u_\infty=0\})=0,\]
		by the continuity of the Lebesgue integral, we have
		\begin{equation}\label{tt1}
			\sum_{i=1}^{N}\int_{Q_{\varepsilon/2}^-(X_i)} u_\infty^{-p}\leq \int_{\mathcal{N}_\varepsilon\cap Q_1^-} u_\infty^{-p}\to0 \quad \text{as}~\varepsilon\to0^+.
		\end{equation}
		Since $X_i\in\{u_\infty=0\}$ and $u_\infty$ is H\"{o}lder continuous,
		\[
		\sup\limits_{Q_{\varepsilon/2}^-(X_i)} u_\infty\leq C \varepsilon^{\alpha}.
		\]
		This implies that
		\begin{equation}\label{tt2}
			\int_{Q_{\varepsilon/2}^-(X_i)} u_\infty^{-p}\geq C \varepsilon^{n+2-\frac{2p}{p+1}}=C \varepsilon^{n+\alpha}.
		\end{equation}
		Plugging \eqref{tt1} into \eqref{tt2} gives us
		\[
		\sum_{i=1}^{N}\varepsilon^{n+\alpha}\leq C \sum_{i=1}^{N}\int_{Q_{\varepsilon/2}^-(x_i,t_i)} u_\infty^{-p}\to 0 \quad \text{as} ~~ \varepsilon\to0^+.
		\]
		Letting $\varepsilon\to0^+$, we conclude that $\mathcal{P}^{n+\alpha}\left(\{u_\infty=0\}\cap Q_1^-\right)=0$.
	\end{proof}
	
	\begin{lem}\label{lemma 4.4}
		The  limit $u_\infty$ satisfies estimates \eqref{es1}-\eqref{es5}.
	\end{lem}
	\begin{proof}
		We only need to show \eqref{es1}-\eqref{es3}. Since $u_i^{-1}\to u_\infty^{-1}$ $\mathcal{P}^{n+2}$ a.e. in $Q_1^-$, these estimates follow immediately from Fatou's lemma and weak convergence of  $\partial_tu_i\rightharpoonup\partial_tu_{\infty}$ and $\nabla u_i\rightharpoonup\nabla u_{\infty}$ in $L^2_{loc}(Q_{3/2}^-)$.
	\end{proof}
	
	\begin{lem}\label{lemma 4.5}
		$u_i^{-p}$ converges to $u_\infty^{-p}$ in $L^1(Q_1^-)$.
	\end{lem}
	\begin{proof}
		By Lemma \ref{lemma 4.3}, for any $\varepsilon>0$ there exists a family of parabolic cylinders $\{Q_{r_k}(X_k)\}$ with $r_k\leq \varepsilon$, which covers $\{u_\infty=0\}\cap Q_1^-$ and satisfies $\sum_k r_k^{n+\alpha}\leq \varepsilon$. For each $k$, choose $Y_k\in\{u_\infty=0\}\cap Q_1^-\cap Q_{r_k}(X_k)$ such that $\mathcal{N}^-:=\cup_kQ_{4r_k}^-(Y_k)$ is an open backward neighborhood of $\{u_\infty=0\}\cap Q_1^-$. Then by the smooth convergence of $u_i$ outside $\{u_\infty=0\}$,
		\begin{equation}\label{ss1}
			\lim\limits_{i\to\infty}\int_{Q_1^-\setminus\mathcal{N}^-} \left|u_i^{-p}-u_\infty^{-p}\right|=0.
		\end{equation}
		Concerning the integral in $\mathcal{N}^-$, inequality \eqref{es1} gives
		\begin{equation}\label{ss2}
			\int_{\mathcal{N}^-} u_i^{-p}\leq \sum_k \int_{Q_{4r_k}^-(Y_k)} u_i^{-p}\leq\sum_k (4r_k)^{n+\alpha}\leq C\varepsilon.
		\end{equation}
		The same inequality holds for $u_\infty^{-p}$ (either by the same argument or simply using Fatou's lemma).
		
		Combining \eqref{ss1} and \eqref{ss2} leads to
		\[
		\limsup_{i\to\infty}\int_{Q_1^-} \left|u_\infty^{-p}-u_i^{-p}\right|\leq C\varepsilon.
		\]
		Sending $\varepsilon\to0$, we complete the proof.
	\end{proof}
	\begin{coro}\label{re 4.6}
		$u_\infty$ is an $L^1$ weak solution of \eqref{eqn} in $Q_1^-$.
	\end{coro}
	This corollary will be used in the following lemma to show the strong convergence of $\nabla u_i$ in $L^2(Q_1^-)$.
	\begin{lem}\label{lemma 4.7}
		As $i\to\infty$, we have
		\[
		\nabla u_i\to \nabla u_\infty~in~L^2(Q_1^-)~\text{and}~u_i^{1-p}\to u_\infty^{1-p}~in~L^1(Q_1^-).
		\]
	\end{lem}
	\begin{proof}
By the elementary inequality $|t^{1-p}-s^{1-p}|\leq C(p)|t-s|(s^{-p}+t^{-p})$,  we  get
		\[
		\begin{split}
			\int_{Q_1^-}|u_i^{1-p}-u_\infty^{1-p}|\leq C \sup\limits_{Q_1^-}|u_i-u_\infty|\int_{Q_1^-}|u_i^{-p}+u_\infty^{-p}|\leq C\sup\limits_{Q_1^-}|u_i-u_\infty|,
		\end{split}
		\]
		which tends to zero by the uniform convergence of $u_i$ to $u_\infty$ in $Q_1^-$.

		Next we take an $\eta\in C^\infty_0(Q_{3/2}^-)$ with $\eta\equiv 1$ in $(Q_1^-)^\circ$. Testing the equation of $u_i$ by $u_i\eta^2$ and integrating by parts, we get
		\[
		\begin{split}
			\int_{Q_{3/2}^-}\left(\partial_t u_i u_i+|\nabla u_i|^2+u_i^{1-p}\right)\eta^2=\int_{Q_{3/2}^-} u_i^2 \Delta\frac{\eta^2}{2}.
		\end{split}
		\]
		Letting $i\to\infty$ yields
		\begin{equation}\label{5.1}
		\lim\limits_{i\to\infty}\int_{Q_{3/2}^-}|\nabla u_i|^2\eta^2+\int_{Q_{3/2}^-}\left(\partial_t u_\infty u_\infty+u_\infty^{1-p}\right)\eta^2=\int_{Q_{3/2}^-} u_\infty^2 \Delta\frac{\eta^2}{2},
\end{equation}
		where we have used the strong convergence of $u_i^{1-p}$ to $u_\infty^{1-p}$ in $L^1_{loc}(Q_{3/2}^-)$ and the weak convergence of  $\partial_tu_i\rightharpoonup\partial_tu_{\infty}$ in $L^2_{loc}(Q_{3/2}^-)$.
		Since $u_\infty$ is a $L^1$ weak solution of \eqref{eqn}, the same process yields
		\[
		\int_{Q_{3/2}^-} |\nabla u_\infty|^2\eta^2+\int_{Q_{3/2}^-} \left(\partial_t u_\infty u_\infty+u_\infty^{1-p}\right)\eta^2=\int_{Q_{3/2}^-} u_\infty^2 \Delta\frac{\eta^2}{2}.
		\]
Combining this equation with \eqref{5.1} gives the strong convergence
		\[
		\lim\limits_{i\to\infty}\int_{Q_1^-}|\nabla u_i|^2=\int_{Q_1^-}|\nabla u_\infty|^2. \qedhere
		\]
	\end{proof}
	
	\begin{coro}\label{coro 4.8}
		The limit function $u_\infty$ is a suitable weak solution of \eqref{eqn} in $Q_1^-$.
	\end{coro}
	\begin{proof}
		Since $u_i$ is stationary in the sense that for for any  $Y\in C_0^{\infty}(\mathbb R^n,\mathbb R^n)$,
		\[
		\begin{split}
			\int \left(\frac{1}{2}|\nabla u_i|^2-\frac{1}{p-1}u_i^{1-p}\right)div Y-DY(\nabla u_i,\nabla u_i)-\partial_tu_i(\nabla u_i\cdot Y)=0,
		\end{split}
		\]
		we let $i\to\infty$ and use the convergence results in the previous lemma to conclude that $u_\infty$ is also stationary. The energy inequality \eqref{localized energy inequality} can be obtained by testing \eqref{localized energy inequality} with $\eta\in C_0^\infty(\R)$ and using these convergence results again. Thus $u_\infty$ is a suitable weak solution of \eqref{eqn} in $Q_1^-$.
	\end{proof}
	This corollary completes the  proof of Theorem \ref{thm 1.6}.

	\section{Monotonicity formula}\label{section 5}
	\setcounter{equation}{0}
	
	In this section,  $u$ always denotes a $C^{\alpha,\alpha/2}$ continuous suitable weak solution of \eqref{eqn} in $Q_2^-$.

	Take an  $\eta\in C_0^{\infty}(B_1(0))$ such that $\eta\equiv 1$ in $B_{1/2}$ and $|\nabla\eta|\leq C$.
	Let
	\[G(y,s):=\left(4\pi s\right)^{-\frac{n}{2}}e^{-\frac{|y|^2}{4s}}\]
	be the standard heat kernel on $\R^n$.
	
	For any $(x_0,t_0)\in Q_1^-$ and $s>0$,  define
	\[
	\begin{split}
		E(s;x_0,t_0,u):=&s^{\frac{p-1}{p+1}}\int_{B_2} \left[\frac{1}{2}|\nabla u(x_0+y,t_0-s)|^2-\frac{1}{p-1}u(x_0+y,t_0-s)^{1-p}\right]\\
		& \qquad \qquad \times G(y,s)\eta(x_0+y)dy\\
		&-\frac{1}{2(p+1)}s^{-\frac{2}{p+1}}\int_{B_2} u(x_0+y,t_0-s)^2G(y,s)\eta(x_0+y)dy.
	\end{split}
	\]
	At present, this is only understood as an integrable function of $s$. In the following, we will also use an averaged version of $E$,
	\[\overline{E}(s;x_0,t_0,u):=\frac{1}{s}\int_{s}^{2s}E(\tau;x_0,t_0,u)d\tau, \]
	which is a continuous function of $s$ and $(x_0,t_0)$.
	
	The main result of this section is the following monotonicity formula.
	\begin{prop}\label{pro 5.1}
		For any $(x_0,t_0)\in Q_1^-$ and $s\in(0,1)$, we have
		\begin{eqnarray*}\label{41}
			E^\prime(s;x_0,t_0,u)&\geq& s^{-\frac{p+3}{p+1}}\int \left(\frac{u}{p+1}-\frac{y\cdot\nabla u}{2}+s\partial_tu\right)^2 G(y,s)\eta(x_0+y) dy\\
			&&- Ce^{-\frac{1}{Cs}}g(s) \nonumber
		\end{eqnarray*}
		in the distributional sense, where  $g(s)$ is an integrable function on $(0,1)$, and in the integral $u$ etc. are evaluated at $(x_0+y,t_0-s)$.
	\end{prop}
	
	\begin{proof}
		Without loss of generality,  assume $(x_0,t_0)=(0,0)$. First by the localized energy inequality \eqref{localized energy inequality},
		\[
		\begin{split}
			E^\prime(s)&\geq \frac{p-1}{p+1}s^{-\frac{2}{p+1}}\int \left(\frac{1}{2}|\nabla u|^2-\frac{1}{p-1}u^{1-p}\right)G\eta\\
			&+s^{\frac{p-1}{p+1}}\int |\partial_t u|^2G\eta +s^{\frac{p-1}{p+1}}\int \partial_tu \left(\nabla u\cdot\nabla G \right) \eta\\
			& +s^{\frac{p-1}{p+1}}\int \partial_tu \left(\nabla u\cdot\nabla  \eta\right) G+\frac{1}{2}s^{\frac{p-1}{p+1}}\int\left(\frac{1}{2}|\nabla u|^2-\frac{1}{p-1}u^{1-p}\right)\partial_sG\eta \\
			&+\frac{s^{-\frac{p+3}{p+1}}}{(p+1)^2}\int u^2G\eta +\frac{s^{-\frac{2}{p+1}}}{(p+1)}\int u\partial_tuG\eta-\frac{s^{-\frac{2}{p+1}}}{2(p+1)}\int u^2\partial_sG\eta.
		\end{split}
		\]
		In the above, $\nabla G$ and $\partial_sG$ can be replaced by
		\begin{equation}\label{43}
			\left\{
			\begin{array}{ll}
				\nabla G=-\frac{y}{2s}G,\\
				\partial_sG=\Delta G=-\frac{1}{2s}\left(n-\frac{|y|^2}{2s}\right)G.
			\end{array}
			\right.
		\end{equation}
		Next we take $Y=yG\eta$ in the stationary condition \eqref{stationary condition}. Because
		\begin{equation}\label{43.0}
			\left\{
			\begin{array}{ll}
				div Y=(n-\frac{|y|^2}{2s})G\eta+\left(y\cdot \nabla\eta\right) G,\\
				DY(\nabla u,\nabla u)=\left[|\nabla u|^2-\frac{(y\cdot\nabla u)^2}{2s}\right]G\eta+(y\cdot\nabla u)(\nabla u\cdot\nabla\eta)G,\\
			\end{array}
			\right.
		\end{equation}
		we get
		\begin{eqnarray*}
			&&\int \left(\frac{1}{2}|\nabla u|^2-\frac{u^{1-p}}{p-1}\right)\left(n-\frac{|y|^2}{2s}\right)G\eta+\int \left(\frac{1}{2}|\nabla u|^2-\frac{u^{1-p}}{p-1}\right)(y\cdot\nabla\eta) G\\
			&=&\int \left[|\nabla u|^2-\frac{|y\cdot\nabla u|^2}{2s}\right]G\eta+\partial_tu (y\cdot\nabla u ) G\eta+\int(y\cdot\nabla u)(\nabla u\cdot\nabla\eta)G.
		\end{eqnarray*}
		Plugging this identity into the above formula for $E^\prime(s)$,  we get
		\[
		E^\prime(s)\geq s^{-\frac{2}{p+1}-1}\int\left(\frac{u}{p+1}-\frac{y\cdot\nabla u}{2}+s\partial_tu\right)^2G\eta dx+ I(s),\]
		where $I(s)$ is an integral involving $\nabla\eta$.  Because $\nabla\eta=0$ in $B_{1/2}$,  we have
		\[
		|I(s)|
		\leq Ce^{-\frac{1}{Cs}}g(s),
		\]	
		where the exponential term comes from the contribution of the heat kernel $G$, and
		\[g(s):= \|u(-s)^{-p}\|_{L^1(B_2)}+\|\partial_tu(-s)\|_{L^2(B_2)}+\|\nabla u(-s)\|_{L^2(B_2)}\]
		is an $L^1((0,1))$ function.
	\end{proof}
	
	This almost monotonicity of $E$ allows us to define the density function
	\[\Theta(x_0,t_0;u):=\lim_{s\to 0^+}E(s;x_0,t_0,u).\]
	\begin{prop}\label{pro 5.2}
		The density function $\Theta$ is upper semi-continuous in the sense that, if $(u_i)$ is a sequence of uniformly $\alpha$-H\"{o}lder continuous, suitable weak solutions of \eqref{eqn} in $Q_2^-$, converging to $u_\infty$ in the sense of Theorem \ref{thm 1.6}, and $(X_i)$ is a sequence of points in $Q_1^-$ converging to $X_\infty$, then
  \[\Theta(X_\infty;u_\infty)\geq \limsup_{i\to\infty}\Theta(X_i;u_i).\]
	\end{prop}
	\begin{proof}
For any $\varepsilon>0$, choose an $s>0$ such that $\overline{E}(s;X_\infty, u_\infty)\leq\Theta(X_\infty; u_\infty)+\varepsilon$. Because $X_i\to X_\infty$ and $u_i\to u_\infty$ in the sense of Theorem \ref{thm 1.6}, we have
		\[
		\overline{E}(s;X_i,u_i)\to \overline{E}(s;X_\infty,u_\infty).
		\]
		Hence
		\[
		\begin{split}
			\limsup_{i\to\infty}\Theta(X_i;u_i)&\leq \lim_{i\to\infty}\overline{E}(s;X_i,u_i)+Ce^{-\frac{1}{Cs}}\\
			&=\overline{E}(s;X_\infty,u_\infty)+Ce^{-\frac{1}{Cs}}\leq\Theta(X_\infty;u_\infty)+\varepsilon+Ce^{-\frac{1}{Cs}}.
		\end{split}
		\]
		Letting $s\to0$ and $\varepsilon\to 0$, we get the desired claim.
	\end{proof}
	
	\begin{prop}\label{pro 5.3}
		$\Theta(x_0,t_0;u)$ is finite if and only if $(x_0,t_0)\in\{u=0\}$. Equivalently, $\Theta(x_0,t_0;u)=-\infty$ if and only if $(x_0,t_0)\in\{u>0\}$.
	\end{prop}
	\begin{proof}
		Assume without loss of generality $(x_0,t_0)=(0,0)$.
		Let us  estimate $\overline{E}(s)$ for $s>0$. Denote
		\[u^s(x,t)=s^{-\frac{1}{p+1}}u(\sqrt{s}x,st), \quad \eta^s(x)=\eta\left(\frac{x}{\sqrt{s}}\right).\]
		By the scaling invariance of $\overline{E}$, we get
		\[
		\begin{split}
			\overline{E}(s)=&\int_{1}^{2}\int_{B_{s^{-1/2}}} t^{\frac{p-1}{p+1}}\left(\frac{1}{2}|\nabla u^s(x,-t)|^2-\frac{u^s(x,-t)^{1-p}}{p-1}\right)G(x,t)\eta^t(x)dxdt\\
			&-\int_{1}^{2}\int_{B_{s^{-1/2}}} \frac{t^{-\frac{2}{p+1}}}{2(p+1)}u^s(x,-t)^2G(x,t)\eta^t(x)dxdt.
		\end{split}
		\]
		We claim that the first integral is always bounded, that is,
		there exists a constant $C$ such that for any $s>0$,
		\begin{equation}\label{upper bound on E,1}
			\int_1^{2}\int_{B_{s^{-1/2}}} t^{\frac{p-1}{p+1}}\left(\frac{1}{2}|\nabla u^s(x,-t)|^2+\frac{u^s(x,-t)^{1-p}}{p-1}\right)G(x,t)\eta^t(x)dxdt\leq C.
		\end{equation}
		This follows from the estimate \eqref{es1} in Lemma \ref{lemma 4.2}, which implies that   for any $x\in B_{s^{-1/2}}$,
		\[\int_{-2}^{-1}\int_{B_1(x)} \left(|\nabla u|^2+u^{1-p} \right)\leq C.\]
		
		If $u(0,0)=0$,  by the $2/(p+1)$-H\"{o}lder continuity of $u$ (and hence $u^s$), we see
		\[u^s(x,t) \leq C\left(1+|x|\right)^{\frac{2}{p+1}}, \quad \forall (x,t)\in B_{s^{-1/2}}\times(-2,-1).\]
		Hence there also exists a constant $C$ such that for any $s>0$,
		\begin{equation}\label{upper bound on E, 2}
			\int_1^{2}\int_{B_{s^{-1/2}}} \frac{t^{-\frac{2}{p+1}}}{2(p+1)}u^s(x,-t)^2G(x,t)\eta^t(x)dxdt\leq C.
		\end{equation}
		Combining \eqref{upper bound on E, 2} and \eqref{upper bound on E,1}, we see in this case
		\[\Theta(0,0;u)=\lim_{s\to 0^+}\overline{E}(s)\geq -C.\]
		
		Next if $u(0,0)=h>0$, then by the H\"{o}lder continuity of $u$, there exists a constant $c>0$ such that
		\[
		u\geq \frac{h}{2} \quad \text{in} ~~ Q_{ch}^-.
		\]
		Hence for all $s<c^2h^2$,
		\[u^s(x,t)\geq \tfrac{h}{2}s^{-\frac{2}{p+1}}  \quad \text{in} ~~ Q_{chs^{-1/2}}^-.\]
		This implies that
		\[\int_{1}^{2}\int_{B_{s^{-1/2}}} \frac{t^{-\frac{2}{p+1}}}{2(p+1)}u^s(x,-t)^2G(x,t)\eta^t(x)dxdt\geq c(h)s^{-\frac{4}{p+1}},\]
		which goes to $+\infty$ as $s\to 0^+$. Combining this fact with \eqref{upper bound on E,1}, we get
		\[\Theta(0,0;u)=\lim_{s\to 0^+}\overline{E}(s)=-\infty. \qedhere\]
	\end{proof}

\section{Stratification of rupture set}\label{section stratification}
\setcounter{equation}{0}
	
In this section, we prove a stratification theorem for the rupture set $\{u=0\}$, where $u$ is an  $\alpha$-H\"{o}lder continuous,  suitable weak solution of \eqref{eqn} in $Q_2$.  We will also show stratification of the time slice $\{u=0\}\cap\{t\}$ for each $t\in\R$. 
 
Our method is essentially  Federer's dimension reduction principle. We will mainly use an abstract version of this principle developed by Brian White, see  \cite[Theorem 2.3 and Theorem 8.2]{White1997}.  We emphasize that the H\"{o}lder regularity and the convergence result established before are crucial for the application of this method to our problem.

\subsection{Blow-up limits}

Because $u$ is a $C^{\alpha,\alpha/2}$-continuous, suitable weak solution of \eqref{eqn} in $Q_2$, it satisfies the a priori estimates \eqref{es1}-\eqref{es5}, and $\Sigma:=\{u=0\}$ is closed.

For any $X_0\in\Sigma$, define the blow up sequence
\begin{equation*}\label{cone 0}
u_{X_0,\lambda}:=\lambda^{-\alpha}u\left(x_0+\lambda x,t_0+\lambda t\right), \quad \lambda\to 0.
\end{equation*}
Thanks to the $\alpha$-H\"{o}lder continuity of $u$, we can use Arzel$\grave{a}$-Ascoli theorem to find a subsequence ($\lambda_i$) and a limit $u_0$ such that $u_{X_0,\lambda_i}:=u_i\to u_0$ uniformly in any compact subset of $\R^{n+1}$.  The limit may not be unique, so we use $\mathcal{T}(u,X_0)$ to denote the set of blow-up limits of $u$ at $X_0$. In the setting of geometric measure theory, $\mathcal{T}(u,X_0)$ is usually called the set of tangent functions.
	
By Theorem \ref{thm 1.6}, we have
\begin{description}
	\item[(i)] $u_i^{1-p}\to u_0^{1-p}$ in $L^1_{loc}(\R^{n+1})$;
	\item[(ii)] $\nabla u_i\to \nabla u_0$ and $\partial_t u_i\rightharpoonup\partial_t u_0$ in $L^2_{loc}(\R^{n+1})$;
	\item[(iii)] $u_0$ is a nonzero suitable weak solution of \eqref{eqn} in $\R^{n+1}$.
\end{description}

For any $X_0=(x_0,t_0)\in\R^{n+1}$ and $s>0$, define
	\[
	\begin{split}
		E(s;X_0,u_0):=&s^{\frac{p-1}{p+1}}\int_{\mathbb R^n} \left(\tfrac{|\nabla u_0(x_0+x,t_0-s)|^2}{2}-\frac{u_0(x_0+x,t_0-s)^{1-p}}{p-1}\right)G(x,s)dx\\
		&-\frac{s^{-\frac{2}{p+1}}}{2(p+1)}\int_{\mathbb R^n} u_0(x_0+x,t_0-s)^2G(x,s)dx.
	\end{split}
	\]
The following result is a simple corollary of the convergence of $u_i$ specified above.
\begin{lem}\label{lem 6.1}
For a.e. $s\geq 0$,
\[\lim_{i\to+\infty}E(s;X_0,u_i)= E(s;X_0,u_0).\]
\end{lem}
	
Next we establish the self-similar property of $u_0$,  which will play an important role in the analysis of the density function $\Theta(\cdot;u_0)$.		
	
\begin{prop}\label{prop backward self-similar}
Let $X_0\in\Sigma$. Then any $u_0\in\mathcal{T}(u,X_0)$ is backward self-similar with respect to $(0,0)$, that is, for any $(x,t)\in\R^n\times(-\infty,0]$ and $\lambda>0$,
\begin{equation}\label{backward self-similar}
u_0(x,t)=\lambda^{-\alpha}u_0(\lambda x, \lambda^2t).
\end{equation}
Moreover, $\Theta(0;u_0)=\Theta(X_0;u)=E(s;0,u_0)$ for any $s>0$.
\end{prop}
	\begin{proof}
		By Proposition \ref{pro 5.1}, for a.e. $0<s_1<s_2<+\infty$,
		\[
		\begin{split}
			&E(s_2;0,u_0)-E(s_1;0,u_0)\\
			&\geq \int_{s_1}^{s_2}t^{-\frac{p+3}{p+1}}\int_{\mathbb R^n}\left[\frac{u_0(x,-t)}{p+1}-\frac{x\cdot\nabla u_0(x,-t)}{2}+t\partial_tu_0(x,-t)\right]^2G(x,t) dxdt.
		\end{split}
		\]
		By Lemma \ref{lem 6.1}, for a.e. $s>0$,
		\[
		\begin{split}
			E(s;0,u_0)=\lim\limits_{i\to\infty}E(s;0,u_i)=\lim\limits_{i\to\infty}E(\lambda_i^2s;X_0,u)=\Theta(X_0;u).
		\end{split}
		\]
		Thus $\Theta(0;u_0)=\Theta(X_0;u)$ and
		\[
		\begin{split}
			\int_{s_1}^{s_2}t^{-\frac{p+3}{p+1}}\int_{\mathbb R^n}\left[\frac{u_0(x,-t)}{p+1}-\frac{x\cdot\nabla u_0(x,-t)}{2}+t\partial_tu_0(x,-t)\right]^2G(x,t) dxdt=0.
		\end{split}
		\]
		Since $s_1$ and $s_2$ are arbitrary, we get
		\[\frac{u_0(x,t)}{p+1}-\frac{x}{2}\cdot\nabla u_0(x,t)-t\partial_tu_0(x,t)=0 \quad \text{a.e. in} ~ \R^n\times(-\infty,0).\]	
		Integrating this equation along $\{x=\sqrt{-t}y\}$  (for a.e. $y\in\R^n$) and then using the continuity of $u$,  we get \eqref{backward self-similar}.
	\end{proof}
	
	\begin{coro}\label{coro 6.3}
		For any  $u_0\in\mathcal{T}(u,X_0)$, $(x,t)\in\R^n\times(-\infty,0]$ and $\lambda>0$,
		\[\Theta(\lambda x,\lambda^2 t;u_0)=\Theta(x,t;u_0).\]
	\end{coro}
	\begin{proof}
		By the self-similarity of $u_0$,  for any $(x,t)\in\mathbb R^n\times (-\infty,0]$ and $\lambda>0$, 
  \[E(\lambda^2 s;\lambda x,\lambda^2 t, u_0)=E(s;x,t,u_0).\]
  Sending $s\to0^+$, we get the desired claim by the definition of $\Theta$.
	\end{proof}

	\begin{prop}\label{prop 6.4}
		For any $u_0\in\mathcal{T}(u,X_0)$ and $X=(x,t)\in\R^{n+1}$, $\Theta(X;u_0)\leq \Theta(0;u_0)$.
	\end{prop}
	\begin{proof}
		Define the blow-down sequence of $u_0$ at $X$ by
		\[(u_0)_{X,\lambda}(Y)=\lambda^{-\alpha}u_0(x+\lambda y, t+\lambda^2s),\]
		where $Y=(y,s)\in\R^n\times(-\infty,0]$ and $\lambda\to+\infty$. Since $u_0$ is backward self-similar with respect to $(0,0)$ and globally H\"{o}lder continuous, we see
		\[(u_0)_{X,\lambda}(Y)=u_0(\lambda^{-1}x+y, \lambda^{-2}t+s)\to u_0(Y)\]
		locally uniformly in $\R^n\times(-\infty,0]$. This uniform convergence  implies that Theorem \ref{thm 1.6} holds for $(u_0)_{X,\lambda}$. Then for a.e. $s>0$,
		\begin{eqnarray}\label{cone 1}
			\Theta(0;u_0)&\equiv E(s;0,u_0)=\lim\limits_{\lambda\to +\infty}E(s;0,(u_0)_{X,\lambda})=\lim\limits_{\lambda\to +\infty}E(\lambda^2 s;X, u_0)\\
			&=\lim\limits_{\tau\to +\infty}E(\tau;X,u_0)\geq \Theta(X;u_0).  \nonumber 
		\end{eqnarray}
  Here the first equality follows from Lemma \ref{lem 6.1}, and in  the last inequality we have used the monotonicity formulat at $X$.
	\end{proof}
	
	\begin{prop}\label{prop max density points}
		If $\Theta(X;u_0)=\Theta(0;u_0)$, then
		\begin{enumerate}
			\item $u_0$ is backward self-similar with respect to $X$;
			\item for any $Y=(y,s)\in \R^n\times(-\infty,0]$ and $\lambda>0$,
			\[\Theta(X+Y;u_0)=\Theta(X+\mathcal{D}_\lambda(Y);u_0);\]
			\item  if $t\le0$, then for any $a\in\R$ and $Y=(y,s)\in\R^n\times(-\infty,0]$,
			\[
			u_0(ax+y,a^2t+s)=u_0(y,s).
			\]
		\end{enumerate}
	\end{prop}
	\begin{proof}
		(1) If $\Theta(X;u_0)=\Theta(0;u_0)$,  \eqref{cone 1} implies that $E(\tau;X,u_0)\equiv\Theta(X;u_0)$, which means $u_0$ is also backward self-similar with respect to $X$.

		(2) This is the same with Corollary \ref{coro 6.3}.

		(3) For any $a>0$, by Corollary \ref{coro 6.3} and (2),
		\[\Theta(\mathcal{D}_{a^{-1}}(X);u_0)=\Theta(X;u_0)=\Theta(0;u_0).\]
		Thus $u_0$ is backward self-similar with respect to $(ax,a^2t)$. Since $u_0$ is also backward self-similar  with respect to $(0,0)$,
		for any $\lambda>0$, we have
		\[u_0(ax+y, a^2 t+s)=\lambda^{-\alpha}u_0(ax+\lambda y, a^2 t+\lambda^2s)=u_0(\lambda^{-1}ax+y, \lambda^{-2}a^2t+s).\]
		Since $u_0$ is continuous, for any $(y,s)\in\R^n\times(-\infty,0]$,
		\[ \lim_{\lambda\to+\infty}u_0(\lambda^{-1}ax+y, (a\lambda)^{-2}t+s)= u_0(y,s).\]
		Hence
		\[u_0(ax+y, a^2 t+s)=u_0(y, s).\]
		A change of variables shows this also holds for $a<0$.
	\end{proof}
	
	\begin{rmk}\label{rmk 6.5}
		According to the definition in \cite[Section 8]{White1997}, $\Theta(\cdot;u_0)$ is a  \emph{backward conelike function}.
	\end{rmk}

	\subsection{Stratification theorem}	
	For any $X_0\in\Sigma$ and $u_0\in\mathcal{T}(u,X_0)$, define
	\[
	\mathcal{M}(\Theta(\cdot;u_0)):=\left\{X\in\R^{n+1}: \Theta(X;u_0)=\Theta(0;u_0)\right\},
	\]
	\[
	\mathcal{V}(\Theta(\cdot;u_0)):=\mathcal{M}(\Theta(\cdot;u_0))\cap\{t=0\},
	\]
	\[
	\mathcal{W}(\Theta(0;u_0)):=\Big\{x\in\R^n: \Theta(y,s;u_0)=\Theta(y+x,s;u_0),~\forall (y,s)\in\R^n\times(-\infty,0]\Big\}.
	\]
	Similar to \cite[Theorem 8.1]{White1997}, by the backward conelike property (i.e. Proposition \ref{prop 6.4} and Proposition \ref{prop max density points}), we have
	\begin{prop}\label{prop 6.7}
		For any $X_0\in\Sigma$ and $u_0\in\mathcal{T}(u,X_0)$, the followings hold.
		\begin{enumerate}
			\item  $\mathcal{V}(\Theta(\cdot;u_0))=\mathcal{W}(\Theta(\cdot;u_0))$, which is a linear subspace of $\R^n$. Moreover, $u_0$ is translation invariant in all directions of this subspace.
			\item Either $\mathcal{M}(\Theta(\cdot;u_0))=\mathcal{V}(\Theta(\cdot;u_0))$ or $\mathcal{V}(\Theta(\cdot;u_0))\times(-\infty,a]$ for some $0\le a\le+\infty$. In the later case $\Theta(\cdot;u_0)$ and $u_0$ are time independent up to $t=a$.
		\end{enumerate}
	\end{prop}

	\begin{defi}\label{def 6.8}
		For any $X_0\in\Sigma$ and $u_0\in\mathcal{T}(u,X_0)$,  the space-time spine of $\Theta(\cdot;u_0)$ is
		\[
		S(\Theta(\cdot;u_0)):=\left\{
		\begin{array}{ll}
			\mathcal{V}(\Theta(\cdot;u_0))\times \R,~~\text{if}~\mathcal{M}(\Theta(\cdot;u_0))=\mathcal{V}(\Theta(\cdot;u_0))\times\R,\\
			\mathcal{V}(\Theta(\cdot;u_0)),~~\text{otherwise}.
		\end{array}
		\right.
		\]
		The dimension of the space-time spine is defined by
		\[
		dim(\Theta(\cdot;u_0)):=\left\{
		\begin{array}{ll}
			dim(\mathcal{V}(\Theta(\cdot;u_0)))+2,~~\text{if}~\mathcal{M}(\Theta(\cdot;u_0))=\mathcal{V}(\Theta(\cdot;u_0))\times\R,\\
			dim(\mathcal{V}(\Theta(\cdot;u_0))),~~\text{otherwise}.
		\end{array}
		\right.
		\]
	\end{defi}
	Notice that $S(\Theta(\cdot;u_0))$ is the largest $\mathcal{D}$-invariant subspace contained in $\mathcal{M}(\Theta(\cdot;u_0))$.
	\begin{defi}
		For any $0\leq k\leq n+2$, define
		\[\Sigma_k=\left\{X_0\in\Sigma: dim(\Theta(\cdot;u_0))\leq k,~\forall u_0\in\mathcal{T}(u,X_0)\right\}.\]
	\end{defi}
	By definition, we have $\Sigma_0\subset \cdots \subset \Sigma_{n+2}$. Our main result on this stratification is
	\begin{thm}\label{thm stratification}
		\begin{enumerate}
			\item  $\Sigma_{n+2}\setminus\Sigma_n=\emptyset$, that is, $\Sigma=\Sigma_n$.
			\item For any $0\leq k\leq n$, $dim_{\mathcal{P}}(\Sigma_k)\leq k$, and $\Sigma_0$ is discrete.
			\item For $\mathcal{L}^1$ a.e. $t$, the set
			\[ S_k(t):=\left\{x: (x,t)\in\Sigma_k\right\}\]
			has Hausdorff dimension at most $k-2$.
		\end{enumerate}
	\end{thm}
	The proof of this theorem is almost the same with the one for \cite[Theorem 8.2]{White1997}. The only property needed to check is the first point. The fact that $\Sigma_{n+2}\setminus\Sigma_{n+1}=\emptyset$ is trivial, because any $u_0\in\mathcal{T}(u,X_0)$ can not be identically zero by \eqref{es4} in Lemma \ref{lemma 4.2}. The fact that  $\Sigma_{n+1}\setminus\Sigma_{n}=\emptyset$ is a consequence of the following lemma.
	\begin{lem}\label{lem no ODE blow up}
		For any $X_0\in \Sigma$ and $u_0\in\mathcal{T}(u,X_0)$,  $\mathcal{M}(\Theta(\cdot;u_0))$ can not be $\R^n\times\{0\}$, $\R^{n-1}\times\{0\}$ or $\R^{n-1}\times\R$.
	\end{lem}
	\begin{proof}
		{\bf Case  1.} If $\mathcal{M}(\Theta(\cdot;u_0))=\R^n\times\{0\}$, then   $u_0(x,0)\equiv 0$. Hence  $u_0$ satisfies
		\begin{equation}\label{cone 10}
			\left\{
			\begin{array}{ll}
				\partial_tu_0-\Delta u_0=-u_0^{-p}~~\text{in}~\R^n\times(0,+\infty),\\
				u_0(x,0)=0~~\text{in}~\R^n.
			\end{array}
			\right.
		\end{equation}
		Because  $u_0^{-p}\in L^1_{loc}(\R^{n+1})$  and $u_0$ is globally H\"{o}lder continuous on $\R^{n+1}$, by Duhamel principle, we have the representation formula
		\[
		u_0(x,t)=-\int_0^t\int_{\R^n}u_0(x-y,t-s)^{-p}G(y,s)dyds, \quad \forall x\in\R^n, ~t>0.
		\]
		As a consequence, $u_0(x,t)<0$ for any $t>0$. This  contradicts the fact that $u_0\geq 0$.
		
		{\bf Case  2.} If $\mathcal{M}(\Theta(\cdot;u_0))=\R^{n-1}\times\{0\}$, by Proposition \ref{prop max density points}, $u_0$ can be viewed as a suitable weak solution of \eqref{eqn} in $\R\times(-\infty,0]$. Because it is backward self-similar (see Proposition \ref{prop backward self-similar}), by \cite[Theorem 10.3.3]{Esposito-Ghoussoub-Guo}, either $u_0=k(-t)^{\frac{1}{p+1}}$ or $w(x):=(-t)^{-\frac{1}{p+1}}u_0(\sqrt{-t}x,t)$ is convex in $x$ for all $|x|$ large enough. For the first possibility, using the fact that $u_0\equiv 0$ at $t=0$, we  get the same contradiction as in Case 1. The second possibility implies that $w$ grows linearly at infinity, which  contradicts with the fact that it is globally $\alpha$-H\"{o}lder continuous.
		
		{\bf Case 3.} If $\mathcal{M}(\Theta(\cdot;u_0))=\R^{n-1}\times\R$,
		by Proposition \ref{prop 6.7}, for any $a\in\R$ and $X, Y\in \R^{n-1}\times(-\infty,0]$, there holds
		\[
		u_0\left(\mathcal{D}_{a^{-1}}(X)+Y\right)=u_0(Y).
		\]
		In particular, $u_0\left(\mathcal{D}_{a^{-1}}(X)\right)=u_0(0)=0$ and thus $u\equiv0$ in $\R^{n-1}\times(-\infty,0]$, which contradicts the fact that
		\[\mathcal{P}^{n+\alpha}\left(\{u_0=0\}\cap Q_1\right)=0,\]
		see Lemma \ref{lemma 4.3}.
	\end{proof}

	\subsection{Stratification of time slices}
	In this subsection, we study the stratification  for time slices of $\Sigma$,
	\[\Sigma_t:=\Sigma\cap\{t\},~\forall t\in\R.\]
	
For this purpose, we need some results on $\Theta(\cdot,0;u_0)$, which are direct consequences of Propositions \ref{prop backward self-similar} and \ref{prop max density points}.
	\begin{lem}\label{lem 6.14}
		For any $x_0\in\Sigma_t$ and $u_0\in\mathcal{T}(u,x_0,t)$, there hold
		\begin{description}
			\item [(i)] $\Theta(\cdot,0;u_0)$ is upper semi-continuous in $x$;
			\item [(ii)] $\Theta(\lambda x,0;u_0)=\Theta(x,0;u_0)$ for any $x\in\R^n$ and $\lambda>0$;
			\item [(iii)] if $\Theta(x,0;u_0)=\Theta(0,0;u_0)$, then  for any $y\in\R^n$ and $\lambda>0$,
			\[\Theta(x+\lambda y,0;u_0)=\Theta(x+y,0;u_0)=\Theta(y,0;u_0);\]
			\item [(iv)] $
			\mathcal{V}(\Theta(\cdot;u_0))=\left\{x: \Theta(x+y,0;u_0)=\Theta(y,0;u_0),~\forall y\in\R^n\right\}$.
		\end{description}
	\end{lem}
		Following the notations in \cite[Section 3]{White1997}, $\Theta(\cdot,0;u_0)$ is a \emph{tangent cone function}, and  $\mathcal{V}(\Theta(\cdot;u_0))$ is  \emph{the translational invariance subspace} of $\Theta(\cdot,0;u_0)$, which is called \emph{the spine} of $\Theta(\cdot,0;u_0)$.

For any $k=0,\dots, n$, define the set
	\[\Sigma_k(t):=\left\{x\in\Sigma_t: dim(\mathcal{V}(\Theta(\cdot;u_0)))\leq k,~\forall u_0\in\mathcal{T}(x,t;u)\right\}.\]	
 Note  that the set $S_k(t)$ in Theorem \ref{thm stratification} is   a subset of $\Sigma_k(t)$.
	
Now we can state the result on the stratification of $\Sigma_t$.
	\begin{thm}\label{thm stratificatoin for slice}
		For any $t$, the followings hold.
		\begin{enumerate}
			\item $\Sigma_{n}(t)\setminus\Sigma_{n-2}(t)=\emptyset$ and thus $\Sigma_t=\Sigma_{n-2}(t)$.
			\item For each $0\leq k\leq n-2$, $dim_{\mathcal{H}}(\Sigma_k(t))\leq k$.
			\item  $\Sigma_0(t)$ is discrete.
		\end{enumerate}
	\end{thm}
	The proof is the same with the one for \cite[Theorem 3.2]{White1997}, except that the first point is a consequence of
	Lemma \ref{lem no ODE blow up}.

	\begin{appendices}
		\section{A Liouville theorem}\label{appendix Liouville}
		\setcounter{equation}{0}
		
		In this appendix, we give a proof of Theorem \ref{thm Liouville}. The proof uses blow down analysis with  the help of Almgren monotonicity formula and the global H\"{o}lder continuity of the solution. 
  
  In this appendix, $u$ denotes a globally $\gamma$-continuous, stationary, two-valued caloric function in $\R^n\times(-\infty,0]$, where $\gamma\in(0,1)$ is a constant. If $\{u=0\}=\emptyset$, $u$ is a positive, globally H\"{o}lder continuous solution to the  heat equation, so by the generalized Liouville theorem, $u$ is a constant. Thus, for the remainder of the proof, we assume $\{u=0\}\neq \emptyset$, say $u(0,0)=0$, and we will prove $u\equiv0$. 
		
		For any $(x_0,t_0)\in\R^n\times (-\infty,0]$ and $s>0$, define
		\begin{equation*}\label{H funcitonal}
			H(s;x_0,t_0):=\int_{\mathbb{R}^{n}}u(x_0+x,t_0-s)^{2}G(x,s) dx,
		\end{equation*}
		and
		\begin{equation*}\label{D functional}
			D(s;x_0,t_0) := s\int_{\mathbb{R}^{n}}|\nabla u(x_0+x, t_0-s)|^{2}G(x,s)dx,
		\end{equation*}
		where $G$ is the standard heat kernel. The frequency is defined to be
		\begin{equation*}\label{N functional}
			N(s;x_0,t_0):=\frac{D(s;x_0,t_0)}{H(s;x_0,t_0)}=\frac{ s\int_{\mathbb{R}^{n}}|\nabla u(x_0+x, t_0-s)|^{2}G(x,s)dx}{\int_{\mathbb{R}^{n}}u(x_0+x,t_0-s)^{2}G(x,s) dx}.
		\end{equation*}

		We first establish the Almgren monotonicity formula. For simplicity, we will assume $(x_0,t_0)=(0,0)$.
		\begin{lem}\label{lem H derivative}
			For any $s>0$,
			\begin{equation}\label{A4}
				H^\prime(s)=\int_{\R^n}2u\left(-\partial_tu+\frac{x\cdot\nabla u}{2s}\right)Gdx=\frac{2D(s)}{s}.
			\end{equation}
			Hereafter $u$ is always evaluated at $(x,-s)$ and $G$ is evaluated at $(x,s)$.
		\end{lem}
		\begin{proof}
			This follows from a direct differentiation and then an integration by parts,  using the fact that $\nabla G(x,s)=-\frac{x}{2s}G(x,s)$.
		\end{proof}
		\begin{coro}\label{coro derivative of log H}
			For any $s>0$, if $H(s)\neq 0$, then
			\begin{equation}
				\frac{d}{ds}\log H(s)=\frac{2N(s)}{s}.
			\end{equation}
		\end{coro}

		\begin{lem}\label{lem D derivative}
			For any $s>0$,
			\begin{equation}\label{A5}
				D^\prime(s)\geq 2s\int_{\R^n}\left(\partial_tu-\frac{x\cdot\nabla u}{2s}\right)^2G dx.
			\end{equation}
		\end{lem}
		\begin{proof}
			Direct calculation gives
			\begin{equation}\label{A6}
				D^\prime(s)=\int_{\R^n} |\nabla u|^{2}G+s\frac{\mathrm{d}}{\mathrm{d} s}\int |\nabla u|^{2}G.
			\end{equation}
			For $m\geq 1$, choose $\rho_m\in C^\infty_0(B_{1+m})$ such that
			\begin{equation*}\label{A7}
				\rho_m=\left\{
				\begin{array}{ll}
					1,~\text{if}~|x|\leq 1,\\
					1+\frac{1-|x|}{m},~\text{if}~1\leq |x|\leq 1+m,\\
					0,~\text{if}~|x|\geq 1+m.
				\end{array}
				\right.
			\end{equation*}
			Clearly, $\rho_m$ converges to the constant function $1$. Let $\eta_m=\rho_m \sqrt{G}$. Substituting $\eta_m$ into the localized energy inequality \eqref{localized energy inequality 2} and letting $m\to+\infty$, we obtain
			\begin{equation}\label{A8}
				\begin{split}
					\frac{\mathrm{d}}{ds}\int |\nabla u|^{2}G &\geq  2\int|\partial_tu|^2G+2\int \partial_tu( \nabla G\cdot\nabla u)-\int \partial_s G|\nabla u|^{2}\\
					&=  2\int|\partial_tu|^2G+2\int \partial_t u \frac{x\cdot\nabla u}{2s}G-\int \Delta  G|\nabla u|^{2}.
				\end{split}
			\end{equation}
	Substituting $Y=xG$ into the stationary condition \eqref{stationary condition 2}, we get
			\begin{equation}\label{A9}
				\begin{split}
					\int\Delta  G|\nabla u|^2 &= s^{-1}\int \left[|\nabla u|^2-\frac{(x\cdot\nabla u)^2}{2s}+(x\cdot\nabla u) \partial_tu\right]G dx.
				\end{split}
			\end{equation}
			As a result, \eqref{A5} follows from \eqref{A6}, \eqref{A8} and \eqref{A9}.
		\end{proof}
		
		\begin{lem}[Almgren monotonicity formula]\label{lem Almgren monotonicity}
			On $(0,+\infty)$, $N(s)$ is non-decreasing in $s$. Moreover, if $N(s)\equiv d/2$ for some $d\geq0$, then $u$ is backward self-similar of degree $d$, i.e., for any $(x,t)\in\R^n\times(-\infty,0]$ and $\lambda>0$,
			\[
			u(x,t)=\lambda^{-d}u(\lambda x,\lambda^2 t).
			\]
		\end{lem}
		\begin{proof}
			Since
			\[N^\prime(s)=\frac{D^\prime(s)H(s)-D(s)H^\prime(s)}{H(s)^2},\]
			it suffices to show $D^\prime(s)H(s)-D(s)H^\prime(s)\geq0$. By Lemma \ref{lem H derivative} and Lemma \ref{lem D derivative},
			\begin{eqnarray*}\label{A10}
				&&D^{\prime}(s)H(s)-D(s)H^{\prime}(s) \\
				&\geq &
				2s\left[\left (\int \left(\partial_tu-\frac{x\cdot\nabla u}{2s}\right)^2 G\right)\left(\int u^{2} G\right)-\left(\int u\left(\partial_{t}u-\frac{x\cdot\nabla u}{2s}\right)G \right)^2\right]\\
				&\geq &0,
			\end{eqnarray*}
			where we have used the Cauchy-Schwarz inequality to deduce the last inequality.
			
			If $N(s)\equiv d/2$, by characterization of the equality case in the Cauchy-Schwarz inequality,  for any $t<0$, there exists an $h(t)$ such that
			\[u(x,t)=h(t)\left[2t\partial_tu(x,t)+x\cdot\nabla u(x,t)\right] \quad \text{for a.e. }~~ x\in \R^n.\]
			This is a differential equation in the time variable $t$. An integration shows that there exists a function $\Psi(t)$ such that for a.e. $x\in\R^n$,
			\begin{equation}\label{A10.1}
				u(x,t)\equiv \Psi(t) u\left(\tfrac{x}{\sqrt{-t}},-1\right).
			\end{equation}
			By the continuity of $u$, this holds for any $x\in\R^n$.

			Finally, substituting \eqref{A10.1} into the equation $N(s)\equiv d/2$, we get
			\[\Psi(t)\equiv (-t)^{\frac{d}{2}},\]
			which implies that $u$ is backward self-similar with respect to $(0,0)$.
		\end{proof}
		\begin{rmk}[Unique continuation property]
			Combining the above Almgren monotonicity formula with Corollary \ref{coro derivative of log H}, we can show as in Caffarelli-Karakhanyan-Lin \cite[Theorem 30]{Caffarelli-Karakhanyan-Lin}   that  unless $u\equiv 0$, then for any $s>0$, $H(s)>0$.
		\end{rmk}
		
		\begin{lem}\label{lem A4}
			For any $(x_0,t_0)\in\{u=0\}$ and $s\geq 0$, $N(s;x_0,t_0)= \gamma/2$.
		\end{lem}
		\begin{proof}
			Assume $(x_0,t_0)=(0,0)$.
			
			{\bf Step 1.} For any $s>0$, $N(s)\leq \gamma/2$.
			
			Assume by the contrary, there exists an $s_0>0$ such that $d/2:=N(s_0)> \gamma/2$. Then by Lemma \ref{lem Almgren monotonicity}, for any $s\geq s_0$, $N(s)\geq d/2$. Hence by Corollary \ref{coro derivative of log H},
			\[
			\frac{d}{ds}\mathrm{log}H(s)=\frac{2N(s)}{s} \geq \frac{d}{s}.
			\]
			Integrating from $s_0$ to $s$, we see for any $s\geq s_0$,
			\begin{equation}\label{A13}
				\frac{H(s)}{H(s_0)}\geq\left(\frac{s}{s_0}\right)^{d}.
			\end{equation}
			However, since $u$ is globally $\gamma$-H\"{o}lder continuous and $u(0,0)=0$, there exists a constant $C>0$ such that
			\[
			|u(x,-s)|^2\leq C(|x|^{2\gamma}+s^\gamma)~~\text{for any}~~x\in\R^n, ~s>0,
			\]
			which implies that for any $s>0$,
			\begin{equation}\label{A12}
				H(s)\leq Cs^\gamma \int_{\R^n}\left(1+\frac{|x|^2}{s}\right)^{\gamma}G(x,s)dx\leq Cs^\gamma.
			\end{equation}
			This is a contradiction with \eqref{A13}  if $s$ is sufficiently large.
			
			{\bf Step 2.} For any $s>0$, $N(s)\geq \gamma/2$.
			
			Assume by the contrary that there exists $s_0>0$ such that  $d/2:=N(s_0)< \gamma/2$. Then by Lemma \ref{lem Almgren monotonicity}, for any $s\in(0, s_0)$, $N(s)\leq d/2$.
			Hence by Lemma \ref{lem H derivative},
			\[
			\frac{d}{ds}\frac{H(s)}{s^d}=-ds^{-d-1}H(s)+2s^{-d-1}D(s)\leq0,
			\]
			which implies that for  $s\in(0, s_0)$,
			\begin{equation}\label{A11}
				H(s)\geq \frac{H(s_0)}{s_0^d}s^d.
			\end{equation}
			But by \eqref{A12}, $H(s)\leq Cs^\gamma$. If $s$ is close to zero, this is a contradiction.
		\end{proof}
		\begin{lem}\label{A.6}
			Either the zero set $\{u=0\}$ is a linear subspace of $\R^n\times\{0\}$ and $u$ is translational invariant in these directions, or $u$ is independent of $t$. 
		\end{lem}
		\begin{proof}
			This is similar to Proposition \ref{prop 6.7}, see \cite[Theorem 8.3]{White1997} for the detailed proof.
		\end{proof}
		
		\begin{proof}[Completion of the proof of Theorem \ref{thm Liouville}]
			If $u$ is independent of $t$, then it is a stationary two-valued harmonic function. By the Liouville theorem in \cite[Appendix A]{Davila-Wang-Wei}, $u$ must be a constant.
			
			
			
			If $\{u=0\}\subset\R^n\times\{0\}$, then
			\[
			\left\{
			\begin{array}{lll}
				\partial_tu-\Delta u=0~\text{in}~\R^n\times (-\infty,0),\\
				u>0~\text{in}~\R^n\times(-\infty,0).
			\end{array}
			\right.
			\]
		Because $u(0,0)=0$,	by strong maximum principle, we obtain $u\equiv0$.
		\end{proof}

		\section{An half space problem}
		\setcounter{equation}{0}
		
		In this appendix we study an half space problem for   stationary two-valued caloric functions.
  \begin{thm}\label{thm half space}
			Suppose $\gamma\in(0,1)$,  $u$ is a bounded, globally $\gamma$-H\"{o}lder continuous, stationary two-valued caloric function in $\R^n_+\times(-\infty,0)$, where $\R^n_+:=\{x_n
			>0\}$. If $u=1$ on $\partial\R^n_+\times(-\infty,0)$, then $u\equiv 1$.
		\end{thm}
		
		First we need a regularity theory for stationary two-valued caloric functions, see \cite{TBD}.
		\begin{prop}\label{prop regularity for two valued function}
			Suppose $u$ is a H\"{o}lder continuous, stationary two-valued caloric function in $Q_1^-$. Then  for any $t\in[-1,0]$, either $u(t)\equiv 0$ in $B_1$ or $u(t)$ cannot vanish in any open subset of $B_1$.
   
Moreover, if it is the later case, then for any $r<1$, in $Q_r^-$, we have
			\begin{description}
				\item [(i) Lipschitz regularity] $u$ is Lipschitz continuous, with the estimate
				\[\sup_{Q_r^-}|\nabla u|\leq \frac{C}{1-r}\sup_{Q_1^-}u;\]
				\item [(ii) Hausdorff dimension]  $\text{dim}_{\mathcal{P}}(\{u=0\})\leq n+1$;
				\item [(iii) Decomposition] $\{u=0\}=\mathcal{R}\cup \mathcal{S}$, where $\mathcal{R}$ is relatively open and $\mathcal{S}$ is closed;
				\item [(iv) Regular set]  $X\in\mathcal{R}$ if and only if there exists an $r>0$ and a caloric function $\Phi$ in $Q_r^-(X)$, with $\Phi(X)=0$ and $\nabla\Phi(X)\neq 0$, such that
				$\{u=0\}\cap Q_r^-(X)=\{\Phi=0\}$, which is a smooth hypersurface, and $u=|\Phi|$ in $Q_r^-(X)$;
				\item [(iv) Singular set]  $X\in\mathcal{S}$ if and only if
				\[\limsup_{r\to 0} \sup_{Q_r^-(X)\cap\{u>0\}}|\nabla u|=0;\]
				\item [(v) Hausdorff dimension for singular set] $\text{dim}_{\mathcal{P}}(\mathcal{S})\leq n$.
				\item[(vi) One dimensional case] If $n=1$, then $\mathcal{S}$ is a discrete set, and $\mathcal{R}$ consists of regular curves with (possible) end points in $\mathcal{S}$.
			\end{description}
		\end{prop}

      The first claim in this proposition is a unique continuation (in spatial directions) property. Applying this result in our setting, because 
     $u$ is continuous and $u=1$ on $\partial\R^n_+\times(-\infty,0)$, we see for any $t\leq 0$, $u(t)$ cannot be identically zero. Thus the second alternative holds for all of $u(t)$. Consequently, the above properties (i)-(vi) hold for $u$.

		\begin{lem}\label{lem B3}
			For any constant vector $\xi\in\R^n$, $|\partial_\xi u|$ is continuous and sub-caloric.
		\end{lem}
		\begin{proof}
			Since $u$ is smooth in $\{u>0\}$, $|\partial_\xi u|$ is continuous here.  Note also $\partial_\xi u$ satisfies the heat equation in this open set, so $|\partial_\xi u|$ is sub-caloric in $\{u>0\}$.
			
			Next we show that $|\partial_\xi u|$ is continuous across $\mathcal{R}$. In fact, for any $X\in\mathcal{R}$, by taking the cylinder $Q_r^-(X)$ and the caloric function $\Phi$ from (iv) in the previous proposition, we see
			\[|\partial_\xi u|\equiv |\partial_\xi \Phi| \quad \text{in} ~~ Q_r^-(X).\]
			Hence $|\partial_\xi u|$ is continuous and sub-caloric in $Q_r^-(X)$.
			
			Finally, by (iv) of the previous proposition, $|\partial_\xi u|$ converges uniformly to $0$ when approaching $\mathcal{S}$. Therefore $|\partial_\xi u|$ is continuous near $\mathcal{S}$.  Now $\mathcal{S}\subset \{|\partial_\xi u|=0\}$, and we have shown that $|\partial_\xi u|$ is sub-caloric outside $\mathcal{S}$, the corollary then follows by an approximation.
		\end{proof}

		With these tools in hand, we first reduce the proof of Theorem \ref{thm half space} to a simpler setting.
		\begin{lem}
			Under the assumptions of Theorem \ref{thm half space}, $u$ must be independent of $x_1$, $\dots$, $x_{n-1}$.
		\end{lem}
		\begin{proof}
			For any $(x,t)\in \R^n_+\times(-\infty,0)$, by the boundedness of $u$, applying (i) of Proposition \ref{prop regularity for two valued function} to $u$ in $Q_{x_n/2}^-(x,t)$, we see there exists a constant $C$ independent of $(x,t)$ such that
			\begin{equation}\label{A2.1}
				|\nabla u(x,t)|\leq Cx_n^{-1}.
			\end{equation}
			
			Because $u\equiv 1$ on $\partial\R^n_+$ and $u$ is globally  $\gamma$-H\"{o}lder continuous, $u>0$ in $\{0<x_n<\delta\}\times(-\infty,0)$ for some $\delta>0$. Thus $u$ is smooth near $\partial\R^n_+\times(-\infty,0)$, and $\partial_{x_i}u\equiv 0$ on $\partial\R^n_+\times(-\infty,0)$.
			
			Take an aribitrary $i\in\{1,\dots, n-1\}$. Denote $\varphi:=|\partial_{x_i}u|$. By Corollary \ref{lem B3}, it is a nonnegative, continuous sub-carloric function. By the discussions in the previous two paragraphs, $\varphi=0$ on $\partial\R^n_+\times(-\infty,0]$ and $\varphi\to0$ uniformly as $x_n\to+\infty$. We want to show that it is identically zero.

			Assume by the contrary that $\varphi$ is nonzero. Then
			\[\sup_{\R^n_+\times (-\infty,0)}\varphi>0.\]
			Take a sequence $(x_k,t_k)$ so that $\varphi(x_k,t_k)$ approaches this supremum. By the strong maximum principle, $t_k\to-\infty$.

			Assume $x_k=(x_{1,k}, \cdots, x_{n,k})$. By \eqref{A2.1}, $x_{n,k}$ is a bounded sequence. Take a subsequence so that $x_{n,k}\to x_{n,\infty}$.
			Let $x_k^\prime=(x_{k,1}, \cdots, x_{k,n-1}, 0)$, $u_k(x,t):=u(x_k^\prime+x,t_k+t)$ and $\varphi_k(x,t):=\varphi(x_k^\prime+x,t_k+t)$. By the global H\"{o}lder continuity of $u$ and Proposition \ref{prop regularity for two valued function}, up to a subsequnce, $u_k$  converges to a limit $u_\infty$ locally uniformly in $\R^n_+\times\R$, and  $\varphi_k=|\partial_{x_i}u_k|$ converges to the limit function $\varphi_\infty=|\partial_{x_i}u_\infty|$.
			
			Since
			\[\varphi_\infty((0, \cdots, x_{n,\infty}),0)=\sup_{\R^n_+\times (-\infty,0)}\varphi_\infty=\sup_{\R^n_+\times (-\infty,0)}\varphi>0,\]
			the strong maximum principle implies that $\varphi_\infty$ is a constant. This is a contradiction with the boundary value condition $\varphi_\infty=0$ on $\partial\R^n_+\times(-\infty,+\infty)$.
		\end{proof}
		
		To finish the proof of Theorem \ref{thm half space}, it remains only to prove the $n=1$ case.
		\begin{proof}[Proof of Theorem \ref{thm half space} in the one dimensional case]
			By \eqref{A2.1}, for any $t\in(-\infty,0]$, the Dirichlet energy
			\[D(t):=\int_0^{+\infty} \partial_xu(x,t)^2 dx\]
			is well-defined. Moreover, because $u$ is globally $\gamma$-H\"{o}lder continuous, there exists a constant $C$ such that
			\[|\partial_xu(x,t)|\leq C \quad \text{in } ~ \left\{0<x<\frac{1}{C}\right\}\times (-\infty,0].\]
			Combining this estimate with \eqref{A2.1}, we see
			\begin{equation}\label{global energy bound}
				\sup_{t\leq 0}D(t)<+\infty.
			\end{equation}
			
			By choosing suitable cut-off functions in the localized energy inequality \eqref{localized energy inequality 2} in the definition of stationary two-valued caloric functions, and noticing that $\partial_t u(0,t)\equiv 0$ (by the boundary condition and the smoothness of $u$ near the boundary), we get a global energy inequality: for any $-\infty<t_1<t_2\leq 0$,
			\begin{equation}\label{global energy inequality}
				D(t_2)-D(t_1)\leq -2\int_{t_2}^{t_1}\int_0^{+\infty} \partial_tu(x,t)^2dxdt.
			\end{equation}
			In view of \eqref{global energy bound}, this implies that
			\begin{equation}\label{A2.2}
				\int_{-\infty}^{0}\left(\int_0^{+\infty} \partial_tu(x,t)^2dx\right)dt<+\infty.
			\end{equation}
			
			For any $t_i\to-\infty$, let
			\[u_i(x,t)=u(x,t_i+t).\]
			By the global H\"{o}lder regularity of $u$ and Proposition \ref{prop regularity for two valued function}, up to a subsequence, $u_i$ converges to a limit $u_\infty$ locally uniformly on $\R_+\times\R$. Here $u_\infty$ is still a  H\"{o}lder continuous, stationary two-valued caloric function. Furthermore, by passing limit in \eqref{A2.2}, we deduce that
			\[\int_{-\infty}^{+\infty}\left(\int_0^{+\infty} \partial_tu_\infty(x,t)^2dx\right)dt=0.\]
			Hence $u_\infty$ is a function of $x$ only. Then condition (ii) in Definition \ref{stationary two-valued caloric fct}  implies that $\partial_{xx}u_\infty=0$ in $\{u_\infty>0\}$. However, since we still have $0\leq u_\infty\leq 1$ in $[0,+\infty)$ and $u_\infty(0)=1$, the only possibility is that $u_\infty\equiv 1$. This limit is independent of the choice of subsequences, so we arrive at the conclusion that
			\begin{equation}\label{A2.3}
				u(\cdot,t)\to 1 \quad \text{locally uniformly on }~ [0,+\infty) ~ \text{as} ~ t\to-\infty.
			\end{equation}
			
			Next  we claim that
			\begin{equation}\label{A2.4}
				\lim_{t\to-\infty}D(t)=0.
			\end{equation}
			To prove this claim, take an arbitrary $R>0$. By \eqref{A2.3}, there exists an $T_R$ such that if $t\leq -T_R$, $u(x,t)$ is very close to $1$ in $[0,2R]$. In particular, $u$ satisfies the heat equation in $(0,R)\times (-\infty,-T_R)$. In view of \eqref{A2.3}, by interior and boundary gradient estimates for the heat equation, we get
			\begin{equation}\label{A2.5}
				\lim_{t\to-\infty}\int_0^R \partial_xu(x,t)^2 dx=0.
			\end{equation}
			For the part on $[R,+\infty)$, we simply use \eqref{A2.1} to get
			\begin{equation}\label{A2.6}
				\int_R^{+\infty}\partial_xu(x,t)^2 dx\leq \frac{C}{R}.
			\end{equation}
			Combining \eqref{A2.5} and \eqref{A2.6}, and then letting $R\to+\infty$, we get \eqref{A2.4}.

			With \eqref{A2.4} in hand, by the monotonicity of $D(t)$ (from \eqref{global energy inequality}), we must have
			\[D(t)\equiv 0 \quad \text{for any} ~ t\leq 0. \]
			Therefore $\partial_xu=0$ a.e. in $\R_+\times(-\infty,0)$. Then $u\equiv 1$ in $\R^n_+\times(-\infty,0]$ because this holds on the boundary and $u$ is continuous.
		\end{proof}
		
	\end{appendices}

\end{document}